\documentclass[a4paper]{article}

\usepackage{
amsmath,
amsthm,
amscd,
amssymb,
}
\usepackage{authblk}
\usepackage{comment}
\usepackage{tikz}
\usetikzlibrary{positioning,arrows.meta,calc}
\usepackage{xspace}
\usepackage{stmaryrd}
\usepackage{wasysym}

\usepackage{pgfplotstable}
\usetikzlibrary{calc}

\usepackage{graphicx}
\usepackage{caption}
\usepackage{subcaption}

\usepackage{mathrsfs}

\usepackage{url}

\usepackage{filecontents}

\usepackage[pdfborder={0 0 0}]{hyperref}

\theoremstyle{plain}
\newtheorem{theorem}{Theorem}[section]
\newtheorem{lemma}[theorem]{Lemma}
\newtheorem{definition}[theorem]{Definition}
\newtheorem{corollary}[theorem]{Corollary}
\newtheorem{proposition}[theorem]{Proposition}
\newtheorem{conjecture}[theorem]{Conjecture}

\newtheorem{remark}[theorem]{Remark}

\newtheoremstyle{derp}
{3pt}
{3pt}
{}
{}
{\upshape}
{:}
{.5em}
{}
\theoremstyle{derp}
\newtheorem{example}[theorem]{Example}

\newcommand{\R}{\mathbb{R}}
\newcommand{\Q}{\mathbb{Q}}
\newcommand{\Z}{\mathbb{Z}}
\newcommand{\C}{\mathbb{C}}
\newcommand{\N}{\mathbb{N}}

\newcommand{\orb}{\mathcal{O}}

\newcommand{\PC}{\mathrm{PC}}
\newcommand{\MPC}{\mathrm{MPC}}

\newcommand{\OC}[1]{\overline{\mathcal{O}(#1)}}

\newcommand\xqed[1]{%
  \leavevmode\unskip\penalty9999 \hbox{}\nobreak\hfill
  \quad\hbox{#1}}
\newcommand\qee{\xqed{$\fullmoon$}}

\newcommand{\supp}{\mathrm{supp\;}}

\newcommand{\Path}{\mathtt{Path}}
\newcommand{\DPath}{\mathtt{DPath}}

\newcommand{\bol}[1]{\mathbf{#1}}

\newcommand{\blab}{cluster}

\newcommand{\blabs}{clusters}

\newcommand{\froctal}{cluster fractal}
\newcommand{\Froctals}{Cluster fractals}
\newcommand{\froctals}{cluster fractals}

\newcommand{\dist}{d}

\makeatletter
\tikzset{
    zero color/.initial=white,
    zero color/.get=\zerocol,
    zero color/.store in=\zerocol,
    one color/.initial=red,
    one color/.get=\onecol,
    one color/.store in=\onecol,
    cell wd/.initial=1ex,
    cell wd/.get=\cellwd,
    cell wd/.store in=\cellwd,
    cell ht/.initial=1ex,
    cell ht/.get=\cellht,
    cell ht/.store in=\cellht,
}

\newcommand{\drawgrid}[2][]{
\medskip
\begin{tikzpicture}[#1]
  \pgfplotstableforeachcolumn#2\as\col{
    \pgfplotstableforeachcolumnelement{\col}\of#2\as\colcnt{%
      \ifnum\colcnt=0
        \fill[\zerocol]($ -\pgfplotstablerow*(0,\cellht) + \col*(\cellwd,0) $) rectangle+(\cellwd,\cellht);
      \fi
      \ifnum\colcnt=1
        \fill[\onecol]($ -\pgfplotstablerow*(0,\cellht) + \col*(\cellwd,0) $) rectangle+(\cellwd,\cellht);
      \fi
    }
  }
\end{tikzpicture}
\medskip
}
\makeatother

\title{Subshifts with sparse traces}

\author{
Ville Salo
}
\affil{University of Turku \\
\url{vosalo@utu.fi}}

\begin{document}
\maketitle

\begin{abstract}
We study two-dimensional subshifts whose horizontal trace (a.k.a. projective subdynamics) contains only points of finite support. Our main result is a classification result for such subshifts satisfying a minimality property. As corollaries, we obtain new proofs for various known results on traces of SFTs, nilpotency and decidability of cellular automata, topological full groups and the subshift of prime numbers. We also construct various (sofic) examples illustrating the concepts.
\end{abstract}

\section{Introduction}

Multidimensional subshifts are $\Z^d$-actions by translations on closed subspaces of $\Sigma^{\Z^d}$. The best-known examples of such are subshifts of finite type (SFTs), which are isomorphic to sets of tilings by square tiles with adjacency contraints. Much of the theory of multidimensional subshifts revolves around the phenomenon that the set of tilings, while defined by finitely many local constraints, can have complex dynamical and computational properties \cite{Be66,Ro71,Ka08,BoPaSc10,HoMe10,DuRoSh12}. While two-dimensional SFTs can be highly complex, they have severe limitations that follow trivially from their definition by local rules -- for example, an SFT containing a nontrivial finite configuration (that is, a configuration of finite support) has positive entropy.

Another popular class are the sofic shifts, images of SFTs through factor maps. These subshifts are already much more general. For example, substitutive subshifts (satisfying some technical conditions) are sofic in two dimensions \cite{Mo89}, and the (horizontal) traces, i.e.\ sets of rows appearing in configurations, of sofic shifts exactly coincide with the class of computable subshifts \cite{Ho09,DuRoSh10,AuSa13}. The latter in particular shows that the theory of multidimensional SFTs, and especially sofics, leads quite naturally to general computable subshifts. While not all computable subshifts are sofic, understanding which of them are is an active research area \cite{GuJe15,OrPa16}, and many simple-looking questions are open.


Cellular automata are another point of view to SFTs, in the sense that in many situations they correspond to SFTs with a rational deterministic direction. Similarly as with SFTs, there is much freedom in constructing cellular automata; many properties turn out to be undecidable, and many kinds of behavior are possible for cellular automata \cite{Ka05,Ka12}.

However, also limitations are known. In this paper, we concentrate on limitations arising from having very simple horizontal traces,\footnote{Of course there are many other known restrictions that force subshifts and cellular automata to behave, for example countability \cite{BaDuJe08,BaJe13}, expansivity of CA \cite{BoKi99} and algebraicity \cite{Sc95}.} and from various types of nilpotency. A well-known result of this type is that the horizontal trace of an infinite SFT cannot be sparse, that is, it cannot consist of only configurations with finitely many nonzero symbols. This is a corollary of the full characterization of possible sofic traces of SFTs given in \cite{PaSc15}.

Other similar results in the nilpotency framework appear in \cite{GuRi10} and \cite{SaTo12c}, where the assumption of SFTness is in some sense replaced with determinism, and the multidimensional subshift is not discussed explicitly. In \cite{GuRi10} it is in particular shown that on certain one-dimensional SFTs, asymptotically nilpotent cellular automata 
are nilpotent, 
and in \cite{SaTo12c} a result of similar flavor is proved for CA on countable sofic shifts, in particular showing that nilpotency is decidable on these subshifts.

This paper is an attempt to clarify and unify such results; our main result gives new proofs for the results of \cite{PaSc15,GuRi10,SaTo12c} mentioned in the previous paragraphs.\footnote{All three papers also contain many other results, which we do not reproduce.} However, our result is much more general, in that it applies to all subshifts with sparse traces, not only SFTs or sofic subshifts. We also contribute two techniques that we find quite universally helpful in structuring such proofs, namely restricting to an almost minimal subsystem and studying paths drawn on configurations.

The corollaries listed above are discussed in more detail in Section~\ref{sec:NewProofs}. In addition to them, we obtain for example results that apply to subshifts where the trace is sparse in an irrational direction and some observations about expansive directions in sparse subshifts. We also extract results for topological full groups, and show that the subshift of prime numbers contains either a finite point or a \froctal{} for purely topological reasons.


We make tiny contributions to the construction side of multidimensional subshifts, and show that general subshifts (and already sofic shifts) with sparse traces can be quite nontrivial. We summarize them in Section~\ref{sec:Constructions}.

\subsection{The main theorem}






In this section, we state the main structure theorems, and need some terminology for points of various forms. A \emph{finite point} is a point whose support is finite (that is, all but finitely many cells contain the symbol $0$). A \emph{highway} (Definition~\ref{def:Highway}) is 
a configuration whose support consists of a single ascending path of bounded width whose movement is guided by a uniformly recurrent sequence. \emph{\Froctals{}} (Definition~\ref{def:Blob}) are recursively defined configurations where finite patterns padded with zeroes (\blabs{}) are collected into larger finite clusters with larger and larger separation.
See Figure~\ref{fig:MainCases} for illustrations.

\begin{figure}
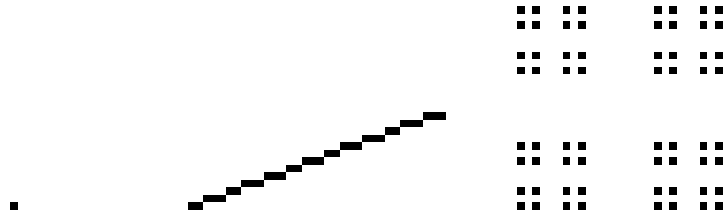

    \centering
    \begin{subfigure}[b]{0.3\textwidth}
    \begin{center}
        \pgfplotstableread{finite.cvs}{\matrixfile}

	   \drawgrid[zero color=white, one color=black, one color=black, cell wd=0.1, cell ht=0.1]{\matrixfile}

        \end{center}
        \caption{A finite configuration: the configuration with a single nonzero symbol.}
    \end{subfigure}
    ~ 
    \begin{subfigure}[b]{0.3\textwidth}
    \begin{center}
        \pgfplotstableread{string.cvs}{\matrixfile}

	   \drawgrid[zero color=white, one color=black, one color=black, cell wd=0.1, cell ht=0.1]{\matrixfile}
	   \end{center}
        \caption{A highway configuration: the discretization of an irrational line.}
    \end{subfigure}
    ~ 
    \begin{subfigure}[b]{0.3\textwidth}
    \begin{center}
        \pgfplotstableread{blobfractal.cvs}{\matrixfile}

	   \drawgrid[zero color=white, one color=black, cell wd=0.1, cell ht=0.1]{\matrixfile}
	   \end{center}
        \caption{A \froctal{}: a subshift based on Cantor's dust.}
        \label{fig:CantorsDust}
    \end{subfigure}
    \caption{Folklore examples of the three cases that can appear in Theorem~\ref{thm:Main}. The trace of the \froctal{} given here is not actually sparse, as we are not aware of such examples in the literature -- see Section~\ref{sec:Examples} for a \froctal{} with a sparse trace.}\label{fig:animals}
    \label{fig:MainCases}
\end{figure}


When studying a minimal subshift, it is useful to express it in a substitutive way, by listing its words of some length, organizing them into longer words that occur in the subshift, each long enough to contain all previous ones, and continuing inductively. This is often formalized by Bratteli-Vershik diagrams \cite{HePuSk92}. 


An \emph{almost minimal} subshift (Definition~\ref{def:AlmoMin}) is one containing a fixed-point for the dynamics, with the property that every point except the fixed one generates the whole subshift. We can give a substitutive structure to such subshift in the following sense (proved as Theorem~\ref{thm:OneDimensionalMainProof}), which is a one-dimensional version of our main theorem:

\begin{theorem}
\label{thm:OneDimensionalMain}
Let $X$ be any almost minimal one-dimensional subshift. Then $X$ is the orbit-closure of a finite point or a \froctal{}.
\end{theorem}


Our main result is Theorem~\ref{thm:Main} below (proved as Theorem~\ref{thm:WhatminimalCases}). It classifies two-dimensional subshifts that are almost minimal, and whose trace is \emph{sparse} (Definition~\ref{def:Sparse}), that is, contains only configurations with finitely many nonzero symbols.

\begin{theorem}
\label{thm:Main}
Let $X$ be any almost minimal two-dimensional subshift with a sparse trace. Then $X$ is the orbit-closure of a finite point, a highway, or a \froctal{}.
\end{theorem}

As a corollary, we find these kinds of configurations in all subshifts where there is suitable directional convergence to a fixed point. The following is our main extraction result of this type, proved as Theorem~\ref{thm:MainFindingProof}:

\begin{theorem}
\label{thm:MainFinding}
Let $X$ be any two-dimensional subshift whose trace $Y$ satisfies one of the following:
\begin{itemize}
\item every point in $Y$ is eventually zero to the right, or
\item $Y$ is countable, and the only periodic point in $Y$ is $0^\Z$.
\end{itemize}
Then $X$ contains a finite point, a highway, or a \froctal{} whose trace is sparse.
\end{theorem}

The way we prove Theorem~\ref{thm:Main} is by taking an arbitrary nonzero point in the subshift $X$ (as the orbit of any such point is dense in the subshift), and attempting to build a \froctal{} structure for the point. This process ends if the point is finite, or if its support contains an infinite path. We then show that every sparse almost minimal subshift where the support of some configuration contains an infinite path in fact contains only uniformly recurrent ascending paths of uniformly bounded width, i.e. highways, which is the most technical part of the proof. 

Intuitively, once we have a configuration containing a path, we first show that it can be made ascending, and then a bounded-width ascending path configuration is found as follows: 
Since the subshift has sparse trace, there is a global limit on the number of parallel paths one can find in a configuration. So we take a maximal number of paths that can be laid parallel to each other (with arbitrarily large separation) in a configuration. Then we observe that if the width of one of these paths is not bounded, we can extract one more path, which is a contradiction. It follows that all the paths must be of bounded width, and we have found our path. By finding a uniformly recurrent point in the orbit-closure of the path, we obtain a highway.

Our formalization of the notion of `path' is quite explicit. Namely, we study dynamical systems of paths as objects in their own right in Section~\ref{sec:Paths}, and then study the path covers of subshifts, where paths are overlaid on top of nonzero symbols of the configuration. This allows a relatively direct translation of the above idea to a proof.

\subsection{Corollaries}
\label{sec:NewProofs}

In this section, we briefly discuss four theorems from the literature that are given new proofs in this paper using Theorem~\ref{thm:Main}. The first one is the following theorem of Pavlov and Schraudner \cite[Theorem~6.4]{PaSc15}:

\begin{theorem}
\label{thm:FullPavlovSchraudner}
If a $\Z$-subshift $Y$ has universal period and is not a finite union of periodic points, then it is not the horizontal trace of any $\Z^2$-SFT $X$.
\end{theorem}

Universal period means that every point is periodic, except possibly at a bounded number of cells.
Reducing this to Theorem~\ref{thm:Main} is straightforward, but somewhat technical, and we do this in Theorem~\ref{thm:FullPavlovSchraudnerProof}. However, the following weaker version of Theorem~\ref{thm:FullPavlovSchraudner} is a direct corollary of Theorem~\ref{thm:Main}:

\begin{theorem}
\label{thm:PavlovSchraudner}
If a nontrivial $\Z$-subshift $Y$ is sparse, then it is not the horizontal trace of any $\Z^2$-SFT $X$.
\end{theorem}

This follows from Theorem~\ref{thm:Main} by restricting to an almost minimal subshift of $X$ 
and observing that horizontal translates of configurations of any of the three types in Theorem~\ref{thm:Main} (finite configurations, highways or \froctals{}) can be freely glued within the SFT to obtain a non-sparse trace (even positive entropy). This is proved as Theorem~\ref{thm:PavlovSchraudnerProof}.

Nilpotency is an important notion in the theory of cellular automata, and it is perhaps the best-known undecidable property for cellular automata on the full shift \cite{Ka92,AaLe74}. The next theorem is due to Guillon and Richard \cite[Theorem~4]{GuRi10}. 
A CA is \emph{nilpotent} if every point is mapped in finite time to the all-zero configuration, and \emph{asymptotically nilpotent} if every configuration tends to the all-zero configuration.

\begin{theorem}
\label{thm:GuillonRichard}
Let $X$ be a one-dimensional transitive SFT and $f : X \to X$ a cellular automaton. Then $f$ is nilpotent if and only if it is asymptotically nilpotent. 
\end{theorem}

This is a direct corollary of the following, proved as Proposition~\ref{prop:TraceLimitSpaceshipProof}, which in turn is a direct corollary of Theorem~\ref{thm:Main}. By a \emph{glider} we mean a non-trivial finite configuration that is shifted by some power of the cellular automaton.

\begin{proposition}
\label{prop:TraceLimitSpaceship}
Let $f : X \to X$ be a cellular automaton on a one-dimensional subshift $X$. If either $f$ is asymptotically nilpotent, the limit set of $f$ is sparse, or the closure of the asymptotic set of $f$ is sparse, then $f$ has a glider.
\end{proposition}

An asymptotically nilpotent CA on an SFT cannot have a glider, as gluing infinitely many gliders together with bounded gaps would clearly give a configuration that does not tend to the all-zero configuration. Thus Theorem~\ref{thm:GuillonRichard} follows from Proposition~\ref{prop:TraceLimitSpaceship}.

We also state two decidability results obtained by the author and coauthors, for which we obtain new proofs. The following is one of the main decidability results in \cite{SaTo12c}.

\begin{theorem}
\label{thm:SaloTorma}
If $X$ is a one-dimensional countable sofic shift, then nilpotency of cellular automata on $X$ is decidable.
\end{theorem}

After characterizing countable sofic shifts as ones consisting of periodic patterns and finitely many transitions between them, this theorem follows from Proposition~\ref{prop:TraceLimitSpaceship} by observing that the existence of gliders for a CA on a sofic shift is semidecidable, and nilpotency is semidecidable.

Finally, we mention an observation made in \cite{BaKaSa16} about the topological full group of a full shift, which follows from Theorem~\ref{thm:Main}, after interpreting the topological full group of a one-dimensional subshift as a two-dimensional subshift in a suitable way. See Theorem~\ref{thm:BarbieriKariSaloProof} for the proof. 


\begin{theorem}
\label{thm:BarbieriKariSalo}
If $X$ is a one-dimensional full shift, then the torsion problem of the topological full group of $X$ is decidable.
\end{theorem}





The one-dimensional version of our main theorem also has some corollaries of interest. We use it to show that one can say something about the possible patterns in prime numbers using almost none of their specific properties. Let $X_{\mathcal{P}}$ be the subshift whose language contains $w \in \{0,1\}^*$ if and only if for infinitely many $n$, we have the equivalence $\forall a \in [0,|w|-1]: n+a \mbox{ is prime} \iff w_a = 1$. See Proposition~\ref{prop:Primes} for the proof.

\begin{proposition}
The subshift $X_{\mathcal{P}}$ contains either a finite point or a \froctal{}.
\end{proposition}

We also give a number-theoretic proof of this fact.

\subsection{Constructions}
\label{sec:Constructions}

In Section~\ref{sec:Examples}, we give some constructions of subshifts with sparse traces and path spaces.


To prove the theorems listed in the previous section, we only use the first two cases of the classification result Theorem~\ref{thm:Main}, as in all those cases \froctals{} turn out to be impossible. However, in general \froctals{} do exist, and in some sense they are the more common case, as the other two cases are in some sense degenerate \froctals{} (finite points are \froctals{} where the separation between \blabs{} is infinite, and infinite paths are \blabs{} of infinite size). We show in particular that subshifts with sparse traces generated by \froctals{} need not contain any nonzero configurations where the diameter of the support of every row is bounded. In fact, this fails as badly as possible, see Example~\ref{ex:UnboundedRows}.

We also discuss the computational side of subshifts with sparse traces. We show that a subshift where the support of every point is an ascending path is $\Pi^0_1$ if and only if it is sofic when the movement speed of the path is bounded, but construct a $\Pi^0_1$ non-sofic subshift with sparse traces. (See Section~\ref{sec:Definitions} or \cite{CeRe98} for definitions.)


Our main technical tool besides almost minimality are uniformly recurrent paths. Paths are (up to a small change in viewpoint) a standard object in the dynamical systems literature, and usually go by the name \emph{cocycle} \cite[Definitions~2.1]{Fo00}. We prove a classification of minimal path spaces, and show that there are four types of such spaces: every path is ascending, every path is descending, every path is bounded, or no path is bounded but some path enters the origin infinitely many times. In Section~\ref{sec:Examples}, we show how to build representative examples of paths in each of these classes, and show that the last class splits further into subclasses.

\section{Definitions}
\label{sec:Definitions}


See \cite{LiMa95} for a reference on symbolic dynamics. A topologically closed shift-invariant subset of $\Sigma^{\Z^d}$ is called a \emph{($\Z^d$-)subshift}, and more generally we call expansive actions on closed subsets of Cantor space subshifts. Two subshifts are \emph{conjugate} if there is a shift-commuting homeomorphism between them. Conjugate subshifts have the same dynamical properties. Cantor space is metrizable, and on occasion we use a metric $\dist$ on this space (the fact $d$ has multiple meanings should not cause confusion).

\emph{Points} are elements of $\Sigma^{\Z^d}$; we also call them \emph{configurations}. A \emph{unary point} is a point $x \in \Sigma^{\Z^d}$ with $\forall \bol v \in \Z^d: x_{\bol v} = a$, for some fixed $a \in \Sigma$. We write this as $a^{\Z^d}$. 
The \emph{shift}\footnote{Our shift convention is a standard one in the symbolic dynamics on abelian groups, where the contents of points are actually shifted in direction $-\bol v$ by the action of the vector $\bol v$. This action is contravariant in general, but it is of course covariant in the abelian case.} or \emph{translation} on $\Sigma^{\Z^d}$ is $\sigma^{\bol v}(x)_{\bol u} = x_{\bol u + \bol v}$. We also call points \emph{configurations}. 


 Most subshifts considered in this article contain the unary point $0^{\Z^d}$, and are \emph{pointed}, in the sense that this \emph{zero point} $0^{\Z^d}$ is thought of as being part of the structure (in particular morphisms, by default, map the zero point to the zero point). 
In particular, the alphabets of our subshifts almost always contain $0$, and in almost all notions where a `zero symbol' is needed, this zero symbol is $0$. 
We write $\Sigma = \Sigma_+ \cup \{0\}$ where $\Sigma_+ \not\ni 0$ is the set of \emph{nonzero symbols}. In this article, a \emph{nontrivial} subshift is one containing at least two points. A \emph{finite point} is one where only finitely many nonzero symbols occur. Such points are also known as \emph{homoclinic}.

A subshift is \emph{SFT} (subshift of finite type) if it is defined by a finite set of forbidden patterns (in the sense that it is the largest subshift of some full shift that misses a defining clopen set), and \emph{sofic} if it is the image of an SFT under a shift-commuting continuous function. 

A \emph{$\Pi^0_1$-subshift} is a subshift $X \subset \Sigma^\Z$ that is on the $\Pi^0_1$ level of the arithmetical hierarchy (equivalently the lightface Borel hierarchy). In other words, there exists a Turing machine that outputs a sequence $w_1, w_2, ...$, $w_i \in \Sigma^*$, such that $X$ is defined by forbidding the words $w_i$. See \cite{CeRe98}.


In the study of nilpotency and pointed subshifts, two particularly important subshifts are the \emph{all-zero} or \emph{trivial} subshift $\{0^\Z\}$, and the subshift $X_{\leq 1} = \OC{\cdots 00 1 00 \cdots}$, called the \emph{sunny-side-up subshift}, \emph{one-one subshift} or \emph{one-or-less subshift}, consisting of two-sided sequences over a binary alphabet with at most one nonzero symbol. The \emph{two-dimensional sunny-side-up} is defined similarly as the binary subshift whose configurations sum to at most one.


A ($d$-dimensional) \emph{pattern} is a function $P : D \to \Sigma$ where $D = D(P) \subset \Z^d$ is a (possibly infinite) subset called the \emph{domain} of $P$. Patterns are closely related to \emph{cylinders}. If $P$ is a pattern, we define the cylinder $[P] = \{x \;|\; x|_{D(P)} = P\}$. However, we more commonly use the terminology that a pattern \emph{occurs} or \emph{appears} in a point, or in another pattern, and by this we mean that a translate of the point is in the cylinder defined by the pattern. We denote this by $P \sqsubset x$. We say \emph{$P$ occurs (or appears) at $\bol v$ in $x$} if $\sigma^{\bol v}(x)|_{D(P)} = P$. This allows the use of descriptive terminology such as `$P$ occurs at a bounded distance from every occurrence of $Q$'. For a subshift $X$ we define $P \sqsubset X \iff \exists x \in X: P \sqsubset x$.

In the one-dimensional case, patterns with a contiguous domain are usually called \emph{words}, and the set of words occurring in a one-dimensional subshift (or a point) is called its \emph{language}. 

Abusing notation, if $P : D \to \Sigma$ is a pattern and $E \subset \Z^d$, we write $P|_E : E \to \Sigma$ for the pattern defined by $P|_E(\bol v) = P|_D(\bol v)$ for $v \in D \cap E$, $P|_E(\bol v) = 0$ otherwise. In particular for a pattern $P : D \to \Sigma$, $P|_{\Z^d}$ is the corresponding configuration in $\Sigma^{\Z^d}$ padded with zeroes.

Two-dimensional patterns with domain $\{i\} \times \Z$ are called \emph{columns} and ones with domain $\Z \times \{i\}$ are called \emph{rows}. We identify rows and columns with one-dimensional configurations in an obvious way. The \emph{(horizontal) trace} of a subshift $X \subset \Sigma^{\Z^2}$ is the set of configurations $y \in \Sigma^\Z$ such that $y$ is a row in a point of $X$, and the row $y$ defined by $y_j = x_{j,i}$ is denoted by $x_i$. The trace is a $\Z$-subshift. We mostly stay in the two-dimensional case, but for $\Z^d$-subshifts, by the trace we mean the one-dimensional subshift whose configurations are seen on the first axis; in particular one-dimensional subshifts are their own traces. On occasion, we refer to traces in other rational directions, and the meaning should be clear (configurations seen along stripes).

The trace is not the same as the dynamical system obtained by considering a subaction by a subgroup of $\Z^2$ -- the trace is a one-dimensional subshift, thus always expansive, but subactions need not be. Note that the trace is not a conjugacy-invariant, but most properties of the trace that are of interest to us -- for example, sparseness, boundedness, countability and points being eventually zero -- are.

A \emph{direction} is an element $\bol v \in \R^2 \setminus \{(0,0)\}$, under the equivalence relation that $\bol v$ and $r \bol v$ are considered equal for all $r \in \R^+ = \{r \in \R \;|\; r > 0\}$. A direction $\bol v \in \R^2$ is \emph{rational} if it is equal (as a direction) to a point in $\Q^2$, equivalently one of $\Z^2$. Note that any direction is defined by a point of the unit circle up to translation, but a rational direction may not be equal to a point of $\Q^2 \cap S^1$.


A \emph{cellular automaton} is  a shift-commuting continuous function $f : X \to X$ on a $\Z$-subshift $X$. We say $f$ is \emph{nilpotent} if $\exists n \in \N: \forall x \in X: f^n(x) = 0^\Z$ and \emph{asymptotically nilpotent} if $f^n(x) \rightarrow 0^\Z$ for all $x \in X$. The \emph{limit set} of $f$ is $\bigcap_{n = 0}^\infty f^n(X)$, and the \emph{asymptotic set} is $\bigcup_{x \in X} \bigcap_{n = 0}^\infty \overline{\bigcup_{m \geq n} f^m(x)}$.


The \emph{spacetime subshift} of a cellular automaton $f : X \to X$ is the $\Z^2$-subshift $Y$ whose trace is the limit set of $f$, and in every configuration,\footnote{Note that the \emph{trace of a cellular automaton} refers to the vertical trace of the spacetime subshift in most references, but this should not cause confusion.} the $(i+1)$th row is the $f$-image of the $i$th row for every $i \in \Z$. A \emph{half-plane in direction $\bol v \in \R^2$} is a half-plane whose boundary is orthogonal to $\bol v$ and whose interior contains all vectors $-r\bol v$ for large enough $r \in \N$. A subshift $X$ is \emph{deterministic} in direction $\bol v$ if, for every half-plane $H \subset \R^2$ in direction $\bol v$, the map $x \mapsto x|_{H \cap \Z^2}$ is injective on $X$.

We sometimes use English words for directions: $(1,0)$ is right, $(0,1)$ is up. With cellular automata, we also use the terminology that application of the CA is `temporal movement', and applying the shift means `spatial' movement. In particular, \emph{spatial periodicity} for $x \in X$ means $\exists n \in \Z: \sigma^n(x) = x$, and \emph{temporal periodicity} means $f^n(x) = x$ for some $x \in X, n > 0$. When we say $x$ is \emph{eventually periodic} we typically refer to spatial periodicity, and mean that $x = \ldots uuu v www \ldots$ for some words $u,v,w$. For such $x$, the minimal choice of $|v|$ is the \emph{preperiod} of $x$.


On $\Z^d$, we use the metric $\dist(\bol u, \bol v) = \max |\bol u_i - \bol v_i|$ and sometimes the $\ell^2$-norm $|\bol u|_2 = \sqrt{\sum \bol u_i^2}$. The zero-vector in every dimension is called $\bol 0$. 
Write $B_r(\bol v)$ for the ball of radius $r$ around $\bol v \in \Z^d$, $B_r(\bol v) = \{\bol u \;|\; \dist(\bol u, \bol v) \leq r\}$, and $B_r = B_r(\bol 0)$. Write $B_r(S) = \bigcup_{s \in S} B_r(s)$. If $P$ is a pattern, write $S(P) \subset D(P)$ for the \emph{support} of $P$, that is, the set of cells in $D(P)$ where $P$ has a nonzero symbol. 

A subset $S$ of $\Z^d$ is \emph{$r$-connected} if the undirected graph with nodes $S$ and edges $\{(\bol u,\bol v) \;|\; d(\bol u,\bol v) \leq r\}$ is connected. The \emph{$r$-components} of a subset of $\Z^d$ are defined as its maximal $r$-connected subsets.

A pattern $P$ is \emph{$r$-padded} if $D(P) \supset B_r(S(P))$, that is, $P$ contains all cells that are at most $r$ away from the support of $P$.
An \emph{$r$-\blab{}}\footnote{These were called ``blobs'' in a previous version.} is an $r$-padded pattern for which
\begin{itemize}
\item the domain's nonzero cells are $r$-connected and
\item the domain's zero cells are all at most $r$ away from the support.
\end{itemize}
Note that a \blab{} is uniquely determined by its support and need not be finite (though we mostly deal with finite \blabs{}). In every configuration, every nonzero cell $i$ is in precisely one $r$-\blab{}, for every $r$: the $r$-\blab{} is the $r$-connected component of the nonzero cells that $i$ belongs to, $r$-padded with $0$-symbols.


The \emph{zero-gluing} of two patterns $P, Q$ is a pattern obtained by taking their `disjoint' union configuration, in the sense that only zero symbols may come from both patterns. More precisely, $R$ is a zero-gluing of $P$ and $Q$ if we have $D(R) = D(P) \cup D(Q)$, $P|_{D(P) \cap D(Q)} = Q|_{D(P) \cap D(Q)} = 0^{D(P) \cap D(Q)}$, and $R|_{D(P)} = P$ and $R|_{D(Q)} = Q$. 
We sometimes write this as $R = P + Q$, and write $x + y$ for the zero-gluing of two points when the full patterns they define on $\Z^2$ can be zero-glued. For an infinite sequence of configurations $x^i$, we also use the notation $\sum_i x^i$ for the limit of the finite gluings $\sum_{i = 1}^k x^i$ when it exists.


A subshift \emph{allows zero-gluing} if there exists $r$ such that when $P$ and $Q$ are $r$-padded patterns that occur in the subshift, and their zero-gluing exists, then the zero-gluing occurs in the subshift. Note that every SFT allows zero-gluing. Not every strongly irreducible subshift allows zero-gluing: On $\Z$, the \emph{even shift} \cite{LiMa95} is such an example, and in higher dimensions we can take the subshift where every row independently contains points from the even shift.


\begin{remark}
The definition of zero-gluing given here differs from that in earlier versions of the paper, where we instead defined zero-gluing as the weaker notion of being able to glue two configurations together as long as their supports are sufficiently separated. Zero-gluing in its weaker form is not equivalent, and the proof of Proposition~\ref{prop:RowsparseZeroGluing} uses the strong notion given above.
\end{remark}


The \emph{essential rowsparseness} of a two-dimensional configuration $x$ is the minimal $m$ such that for some fixed $r$ the support $S \subset \Z$ of every row of $x$ is contained in the union of $m$ many $r$-balls, that is, $\exists n_1, \ldots, n_m : S \subset \bigcup_i n_i + B_r(\bol 0)$. 


\begin{definition}
\label{def:Blob}
A configuration $x \in (\Sigma_+ \cup \{0\})^{\Z^d}$ is a \emph{\froctal{}} if it is nontrivial and there is an increasing sequence $r_1, r_2, r_3, \ldots$ of natural numbers and for all $r_i$ a finite set $B_i$ of finite $r_i$-\blabs{} such that for all $i$,
\begin{itemize}
\item each \blab{} in $B_{i+1}$ is obtained by zero-gluing translates of \blabs{} in $B_i$ and adding the $r_{i+1}$-padding,
\item each \blab{} in $B_{i+1}$ contains a translate of every \blab{} in $B_i$, 
\item each \blab{} in $B_{i+1}$ contains at least two translates of \blabs{} in $B_i$ with disjoint supports, and
\item $x = \lim_{j \rightarrow \infty} \sigma^{\bol v_j}(b_j|_{\Z^d})$ for some sequences $\bol v_j \in \Z^d$, $b_j \in B_j$. 
\end{itemize}
\end{definition}



\section{Almost minimality and traces}


A subshift is minimal if it is the orbit closure of every point that it contains. Because trace sparse subshifts (see Definition~\ref{def:Sparse}) always contain the all-zero point by compactness, they clearly cannot be minimal (unless trivial). The next best thing is almost minimality -- every point except the all-zero point generates the subshift in its orbit closure. It turns out that we can generally extract almost minimal subshifts from our subshifts of interest. 



\begin{definition}
\label{def:AlmoMin}
Let $X$ be a zero-dimensional compact metric space and let $G$ be a finitely-generated discrete group acting on $X$ by a continuous map $\sigma : G \times X \to X$. Then $(X,G,\sigma)$ is called a \emph{$G$-system}. The \emph{subsystem poset} $\mathcal{S}(X)$ of $X$ is the poset whose elements are subsystems of $X$ with order $Y \leq Z \iff Y \subset Z$. We say $X$ is \emph{minimal} if it is an atom of this poset, that is, $\mathcal{S}(X) = \{\emptyset, X\}$, and \emph{almost minimal} if $\mathcal{S}(X) = \{\emptyset, Y, X\}$ for some one-point system $Y = \{y\} \neq X$.
\end{definition}


\begin{remark}
Note that in our definition a trivial system is not almost minimal.
Almost minimality is discussed in \cite{Da01}.  We called systems with finite $\mathcal{S}(X)$ \emph{quasiminimal} in \cite{Sa14b}. Systems with a unique minimal subsystem are called \emph{essentially minimal} \cite{HePuSk92}. Clearly almost minimal subshifts are essentially minimal and quasiminimal, but neither essential minimality nor quasiminimality implies almost minimality by easy examples, and essential minimality and quasiminimality are orthogonal concepts. 
\end{remark}


\begin{definition}
\label{def:Sparse}
A subshift $X \subset \Sigma^{\Z^d}$ is \emph{$k$-sparse} if every configuration $x \in X$ satisfies
\[ |\{\bol v \in \Z^d \;|\; x_{\bol v} \neq 0\}| \leq k. \] 
A subshift is \emph{sparse} if it is $k$-sparse for some $k \in \N$. A $\Z^d$-subshift is \emph{(nonuniformly) trace sparse} if its trace is sparse, and \emph{uniformly trace sparse} if its trace is $k$-sparse for some $k$. We say that a $\Z^d$-subshift is \emph{trace countable} if its trace is a countable subshift.
\end{definition}

A (trace) countable subshift obviously need not be (trace) sparse. A (trace) sparse subshift need not be uniformly (trace) sparse:

\begin{example}
\label{ex:NonuniformlyRowsparse}
There is a sparse $\Z$-subshift that is not uniformly sparse: Let $x^n \in \{0,1\}^\Z$ be the point where $x^n_i = 1 \iff i \in n\Z \cap [0,n^2]$. Then $\OC{\{x^n \;|\; n \in \N\}}$ is sparse, but not uniformly sparse. \qee
\end{example}

In the next sections, we give some technical tools that are helpful when applying the main theorem outside the almost minimal trace sparse setting-- namely we give sufficient conditions that allow us to extract almost minimal trace sparse subshifts from a multidimensional subshift. In Section~\ref{sec:Extracting}, we give sufficient conditions for the extractability of an almost minimal subsystem. In Section~\ref{sec:Rowsparsing}, we show that among almost minimal subshifts, trace sparseness is implied by various weaker properties. These conditions are somewhat different, and their intersection gives the following lemma. A configuration $x \in \Sigma^{\Z}$ is \emph{eventually zero to the right} if $x_i = 0$ for all large enough $i \in \N$.

\begin{lemma}
\label{lem:Extract}
Let a subshift $X \subset \Sigma^{\Z^2}$ with trace $Y$ satisfy one of the following:
\begin{itemize}
\item $Y$ is sparse,
\item $Y$ is countable and the only periodic point in $Y$ is $0^\Z$, or
\item every point of $Y$ is eventually zero to the right.
\end{itemize}
Then $X$ contains an almost minimal uniformly trace sparse subsystem.
\end{lemma}


\begin{proof}
If $Y$ is sparse, apply any of the four items of Lemma~\ref{lem:RowsparseContainsAlmostMinimal} to obtain an almost minimal subsystem $Z \subset X$, and then apply Lemma~\ref{lem:IsolPeriodic} or Lemma~\ref{lem:EventuallyZeroRight} to conclude $Z$ is uniformly trace sparse.

If $Y$ is countable and the only periodic point is $0^\Z$, apply item 1 of Lemma~\ref{lem:RowsparseContainsAlmostMinimal} to obtain an almost minimal subsystem $Z \subset X$, and then apply Lemma~\ref{lem:IsolPeriodic} to conclude that $Z$ is uniformly trace sparse. For this, observe that the only periodic point $0^\Z$ cannot be isolated in the trace of $Z$ as $0^{\Z^2}$ is in the orbit closure of every point of $Z$.

If every point of $Y$ is eventually zero to the right, apply item 2 or item 4 of Lemma~\ref{lem:RowsparseContainsAlmostMinimal} and then Lemma~\ref{lem:EventuallyZeroRight}.
\end{proof}

\subsection{Extracting an almost minimal subsystem}
\label{sec:Extracting}


\begin{lemma}
\label{lem:EssentiallyImpliesAlmost}
Let $X$ be an essentially minimal $\Z^d$-subshift containing the point $0^{\Z^d}$. Then $X$ contains an almost minimal subsystem.
\end{lemma}

\begin{proof}
Since $\{0^{\Z^d}\}$ is minimal, it is the unique minimal subsystem. 
If $(X_i)_{i \in I}$ is a decreasing net of subshifts containing a nonzero symbol, then by the pigeonhole principle some nonzero symbol $a \in \Sigma_+$ occurs in these subshifts for arbitrarily large $i \in I$, and thus for all $i \in I$. The intersection $C_i = X_i \cap [a]$ is a nonempty compact set for all $i$, and $(C_i)_{i \in I}$ is then a decreasing net of nonempty closed sets. By compactness there exists $x \in \bigcap_i C_i \subset X_i$, so also the intersection contains a nonzero symbol. The result follows from Zorn's lemma.
\end{proof}

\begin{lemma}
\label{lem:RowsparseContainsAlmostMinimal}
Let a subshift $X \subset \Sigma^{\Z^d}$ with trace $Y$ satisfy one of the following:
\begin{itemize}
\item $Y$ is countable and the only periodic point in $Y$ is $0^\Z$,
\item every point of $Y$ is eventually zero to the right
\item the upper Banach density of nonzero symbols is $0$ in every configuration of $Y$, or
\item every point of $X$ contains arbitrarily large balls of zeroes.
\end{itemize}
Then $X$ contains an almost minimal subsystem.
\end{lemma}

\begin{proof}
Let us show that the first condition implies the fourth (so that each of the first three conditions implies the last one). Every subsystem of $Y$ is countable, and thus contains a minimal subsystem that must be countable. Applying this to orbit closures of points, it follows that every point has $0^{\Z}$ in its orbit closure, and thus contains arbitrarily long words over $0$. Since $Y^k$ is countable for all $k$, it also has only the periodic point $(0^{\Z})^k$, and we obtain the fourth condition.

If $Z \subset X$ is minimal, and contains a configuration with a nonzero symbol, then that symbol appears with bounded gaps (in a syndetic subset of coordinates) in every configuration of $Z$. This clearly contradicts the last condition. It follows that if any of the four conditions holds, the only minimal subsystem of $X$ is $\{0^{\Z^d}\}$, so we can apply the previous lemma.
\end{proof}

One can strengthen this lemma to saying that every nontrivial subsystem of $X$ with one of these properties has an almost minimal subsystem (though it may not be the same almost minimal subsystem -- consider any union of two almost minimal subsystems sharing the fixed point).


One can ask, then, to what extend this stronger version has a converse, that is, if every subsystem of $X$ has an almost minimal subsystem, which of the conditions does it satisfy? The fourth condition is clearly also necessary: if there is a configuration not containing a ball of zeroes of radius $r$, then its orbit closure cannot be almost minimal (with fixed point $0^{\Z^2}$). The first two conditions are not necessary:

\begin{example}
There exists a subshift $X \subset \Sigma^{\Z^2}$ which is almost minimal and whose trace is $\{0^\Z, 1^\Z\}$. Namely, take any one-dimensional almost minimal subshift $Y$ (for example, a subshift generated by a finite point) and consider the two-dimensional subshift with constant rows, and columns taken from $Y$. \qee
\end{example}

Slightly more interestingly, one can also build almost minimal subshifts where the upper Banach density of nonzero symbols is positive in every nonzero point of the trace (see Section~\ref{sec:Examples}).

The first two conditions of Lemma~\ref{lem:RowsparseContainsAlmostMinimal} are not comparable for general subshifts: The subshift that is the orbit closure of the characteristic function of $\{\pm 2^{i} \;|\; i \in \N\}$ is countable, has only the periodic point $0^\Z$, and contains points that are not eventually zero in either direction. Conversely, for every $y \in \{0,1\}^\N$ construct the point $x^y \in \{0,1\}^\Z$ with support $\{-2^{n+1} + y_n \;|\; n \in \N\}$. The set of these points generates an uncountable subshift where every point is eventually zero to the right.

\subsection{Trace sparseness of almost minimal subshifts}
\label{sec:Rowsparsing}

\begin{lemma}
\label{lem:IsolPeriodic}
Let $X$ be an almost minimal $\Z^2$-subshift with countable trace $Y$ such that $Y$ has no isolated periodic points. 
Then $X$ is uniformly trace sparse.
\end{lemma}


\begin{proof}
Let $Y$ be the trace of $X$. Because $Y$ is countable, it is not perfect, and thus contains an isolated point $y$ with isolating pattern $w$ which appears in a point of $X$. By the assumption, $w$ is not all zero, since $0^\Z$ is a non-isolated periodic point in $Y$. Identify $w$ with the corresponding two-dimensional pattern of shape $|w|$-by-$1$. By almost minimality, there is a bound $C > 0$ such that whenever a point $x \in X$ has a nonzero symbol $x_{(i,j)} \neq 0$, $w$ appears at $x_{(i',j')} \in \Z$ for some $d((i',j'), (i, j)) < C$. this forces the row $x_{j'}$ to be a translate of $y$.

If $z \in Y$ has a support of at least size $(2C+1)^2$, and $z$ is a row in $x$, then at a bounded distance of each of the many nonzero symbols of $z$ in $x$ the isolating pattern $w$ appears. In particular by the pigeonhole principle, two occurrences of $w$ are on the same row, and thus $y$ must be periodic, since we find two occurrences of the isolating pattern $w$ in it, at a nonzero offset. But then $y$ is an isolated periodic point, which contradicts the assumption. Thus $Y$ must be uniformly sparse.
\end{proof}




\begin{lemma}
\label{lem:EventuallyZeroRight}
Let $X$ be an almost minimal subshift with trace $Y$ such that every point of $Y$ is eventually zero to the right. Then $X$ is uniformly trace sparse.
\end{lemma}


\begin{proof}
If every word $w \sqsubset Y$ can be extended to the right by a word $0^ka$ with $a \neq 0$, then clearly $Y$ contains a configuration that is not eventually zero to the right. Thus, there is a word $w$ to the right of which only an infinite word of all zeroes can occur. Such a word occurs at a bounded distance from every nonzero symbol of every configuration of $X$, and we see as in the proof of the lemma above that $X$ is uniformly trace sparse.
\end{proof}

\section{Paths and path covers}
\label{sec:Paths}

\subsection{Paths}


Let $\Path_r$ be the space of all functions $p : \Z \to \Z$ such that $p(0) = 0$ and $|p(i+1)-p(i)|\leq r$ for all $i$, and $\Path = \bigcup_{r \in \N} \Path_r$. We give $\Path$ the product topology and the shift action $\tau : \Path \to \Path$ defined by $\tau(p)(i) = p(i+1) - p(1)$ (which is a homeomorphism). Let $\DPath_r$ be the space of all functions $p : \Z \to \Z$ with $p(\Z) \subset [-r, r]$ and $\DPath = \bigcup_{r \in \N} \DPath_r$. We give it the product topology and the usual shift action, also denoted $\sigma$. The two spaces are conjugate by the discrete derivative $\phi : \Path \to \DPath$ defined by $\phi(p)(i) = p(i+1) - p(i)$.

As a dynamical system, $\DPath$ is just a direct union of full shifts, and $\Path$ is then also conjugate to such a union. As the name implies, we think of an element of $\Path$ as a two-way infinite path in $\Z$, and $\DPath$ is the sequence of moves that it follows.

All of the following definitions will apply to both kinds of paths, in the sense that if we have defined a property P for elements of $\Path$ we will say a path $p \in \DPath$ has property P if $\phi^{-1}(p)$ has property P, and vice versa.

An \emph{ascending path} is a path $p \in \DPath$ such that for some $m \in \N$, $\sum_{i = j}^{j+m} p(i) > 0$ for all $j$, and symmetrically, we define \emph{descending paths}. A path $p \in \Path$ is \emph{bounded-to-one} if for some $m \in \N$, $|p^{-1}(n)| \leq m$ for all $n \in \Z$. A \emph{finite-to-one} path is a path $p \in \Path$ with $p^{-1}(n)$ finite for all $n \in \Z$, and otherwise $p$ is \emph{infinite-to-one}. A \emph{bounded path} is a path $p \in \Path$ with $p(\Z)$ contained in some $[-m, m]$. A set of paths $X \subset \DPath$ is \emph{uniformly ascending} (resp. descending, bounded-to-one, bounded) if every path in $X$ has the property, and the constant $m$ can be chosen independently of the path. In the case of ascending paths, we call this the \emph{ascension constant}.

Uniform recurrence and minimality, for paths and systems of paths, mean the usual dynamical notions, and in $\DPath$ have the same interpretations as the usual ones for subshifts. Note that an ascending path can be recurrent, since recurrence means recurrence of the sequence of moves, not the sequence of cells the path visits.

In the two-dimensional case, we define analogous systems $\Path^2$ and $\DPath^2$ of paths $p : \Z \to \Z^2$. For such paths, we define ascending and descending paths differently: to a path $p : \Z \to \Z^2$, we associate its \emph{height path} $p' : \Z \to \Z$ by forgetting the first coordinate of $p(n)$ for all $n$. An \emph{ascending path} $p : \Z \to \Z^2$ is one where the height path $p' : \Z \to \Z$ is ascending, and \emph{descending paths} are defined symmetrically.

A \emph{($\Sigma$-)colored path} is $p : \Z \to \Z \times \Sigma$ or $p : \Z \to \Z^2 \times \Sigma$ for a finite set of colors $\Sigma$. We write $\Path_\Sigma$ and $\Path^2_\Sigma$ for the spaces of $\Sigma$-colored path. All notions defined above generalize in an obvious way to colored paths (for geometric definitions we forget the colors, and for dynamical ones, we include them).

Especially in informal geometric explanations, we will also discuss finite subpaths of infinite paths, which need not go through the origin. These are simply functions $p : \{a,...,b\} \to \Z$ or $p : \{a, ..., b\} \to \Z^2$ for some $a, b \in \Z$.

We give a classification theorem for minimal path spaces.

\begin{theorem}
\label{thm:PathCharacterization}
Let $X \subset \Path_r$ be a minimal path space. Then exactly one of the following holds:
\begin{itemize}
\item every path in $X$ is uniformly ascending,
\item every path in $X$ is uniformly descending,
\item every path in $X$ is uniformly bounded, or
\item every path in $X$ is unbounded, and some path in $X$ is infinite-to-one.
\end{itemize}
\end{theorem}

\begin{proof}
A word $w \sqsubset X$ is a \emph{cut path} if whenever $p \in X$, $p_{[0,|w|-1]} = w$, we have $p_j \notin [0,r-1]$ for all $j \notin [0,|w|-1]$. It is easy to verify that if $X$ has a cut path, then we have one of three cases: One possibility is that in every point, to the right of every cut path, the path stays above the origin, and stays below it on the left side, in which case by the syndeticity of $w$ (by minimality) every point is ascending. Symmetrically, if every path is above the origin on the left and below it on the right, every path is descending. In the remaining case, there is a cut path that forces the rest of the path to stay below or above it. This is easily seen to contradict minimality. 

Suppose then that $X$ does not contain a cut path. Then it is easy to construct a path that is infinite-to-one: start with an arbitrary finite path $w_1$. Since it is not a cut path, we can continue it to a longer path $w_2$ that re-enters $[0,r-1]$. Continuing inductively, in the limit we obtain a path that enters one of the coordinates in $[0,r-1]$ infinitely many times. Thus, we are in the third or the fourth case. What remains is to show that if not every path in $X$ is bounded, then every path in $X$ is unbounded. This holds since not being bounded by $m$ is an open condition and $X$ is minimal.
\end{proof}

Path spaces of the last type are not encountered in the proofs of our main results, but it is perhaps the most interesting case, and it splits into further subtypes. The present classification is enough for our purposes, but for the interested reader, we give some examples of the different possible types of class four minimal path spaces in Section~\ref{sec:Examples}.

We need the following generalization in the proof of Proposition~\ref{prop:DeterministicContainsSinglyPeriodic}. The proof is exactly the same.\footnote{Note that the possible finite sequences of moves still form a discrete set, even though paths themselves may explore a dense set of positions.}


\begin{theorem}
\label{thm:PathCharacterizationGeneral}
Let $R \subset [-r,r]$ be a finite set of real numbers and let $X \subset R^\Z$ be a minimal subshift. Define
\[ Y = \{ f : \Z \to \R \;|\; f(0) = 0 \wedge \exists x \in X: \forall i: f(i+1) = f(i) + x_n \}. \]
Then with the obvious definitions, exactly one of the following holds:
\begin{itemize}
\item every path in $Y$ is uniformly ascending,
\item every path in $Y$ is uniformly descending,
\item every path in $Y$ is uniformly bounded, or
\item every path in $Y$ is unbounded, and some path $p \in Y$ satisfies $p_i \in [0,r]$ for infinitely many $i$.
\end{itemize}
\end{theorem}

\subsection{Path covers}
\label{sec:PathCovers}

The \emph{(acyclic) $r$-path cover} of a subshift $X \subset \Sigma^{\Z^d}$ is the subshift $\PC_r(X) \subset X \times (\{\#\} \cup [-r,r]^d)^{\Z^d}$ of configurations $(x,y)$ such that
\[ y_{\bol u} \neq \# \implies x_{\bol u} \neq 0 \wedge y_{\bol u + y_{\bol u}} \neq \# \]
for all $\bol u \in \Z^d$, such that in the directed graph $G(y)$ with nodes $\{\bol u \;|\; y_{\bol u} \neq \#\}$ and edges $\{(\bol u, \bol u + \bol v) \;|\; \bol v \neq \bol 0, y_{\bol u} = \bol v, y_{\bol u + \bol v} \neq \#\}$, every node has in-degree at most one (the out-degree is automatically also at most one), and there are no cycles. In $G(y)$, every node is part of a unique maximal (finite, one-way infinite or two-way infinite) path. If $(x,y) \in \PC_r(X)$ is such that $\bol 0 \in G(y)$, the path through the origin $\bol 0$ is two-way infinite and all nodes of $G(y)$ are on this path, then $(x,y)$ is a \emph{marked path configuration}, and we write $\MPC_r(X)$ for the set of such configurations, which is a compact (possibly empty) space.

\begin{lemma}
\label{lem:MPC}
Let $X$ be a subshift. If $r$ is such that for every $n$, $X$ has a configuration whose support contains an $r$-component of size at least $n$, then $\MPC_r(X)$ is nonempty.
\end{lemma}


\begin{proof}
Suppose that $X$ has, for all $n$, some configuration $x$ whose support has an $r$-component of size at least $n$. Then there are arbitrarily long distances between cells in the arbitrarily large components, so consider an $r$-path of minimal length between two cells at distance at least $2n + 1$ along the support of $x$, and construct a pair $(x,y) \in \PC_r(X)$ where a path between them is drawn on the $y$-component. Translating such pairs so that the middle of the path is at the origin, we obtain configurations in $\PC_r(X)$ where the maximal path through $\bol 0$ is of length at least $n$ in both directions. Any limit point of such pairs as $n \rightarrow \infty$ gives a point of $\MPC_r(X)$.
\end{proof}

See Section~\ref{sec:PathCoverExample} for an example of such a path extraction.

Let us now concentrate on the case $d = 2$. The importance of the space $\MPC_r(X)$ is that it inherits the dynamics of $\Path^2$ in an obvious way, the action of $n \in \Z$ following the path written on the second component for $n$ steps. More precisely, if $(x,y) \in \MPC_r(X)$ and $y_0 = \bol v$, define $\tau(x,y) = \sigma^{\bol v}(x,y)$, so that $(\MPC_r(X), \tau)$ becomes a $\Z$-system.



\begin{definition}
\label{def:Highway}
An \emph{$r$-road (configuration)} $x$ is one whose support, up to translation, is the range of a two-way infinite ascending $r$-path $p : \Z \to \Z^2$. If we can choose $p$ so that the corresponding colored path $q$ defined by $q(i) = (p(i), x_{p(i)})$ is uniformly recurrent, then we call the configuration $x$ an \emph{$r$-highway}. 
An \emph{$r$-preroad (resp. $r$-prehighway) configuration} is one whose nonzero support contains a translated image of an infinite ascending (resp. ascending and uniformly recurrent) $r$-path. 
\end{definition}

The essential trace sparseness of a road configuration is~$1$, and the essential trace sparseness of a preroad is at least $1$, but need not be finite. Note that $r$-prehighways do not, in the general case, have anything to do with paths -- the all-$1$ point is an $r$-prehighway for every $r$. 


\begin{lemma}
Let $X$ be a trace sparse subshift. If $\MPC_r(X)$ is nonempty, then $X$ contains an $r$-preroad.
\end{lemma}

\begin{proof}
Take any minimal subsystem of $\MPC_r(X)$ in the $\Z$-action $\tau$ and a point $(x,y)$ in it. Then the support of $y$ is the range of a uniformly recurrent path $p \in \Path^2$. It follows from Theorem~\ref{thm:PathCharacterization} that $p$ is either ascending, descending or not finite-to-one. The last case is impossible because $X$ is trace sparse. If the path is descending, its reversal is ascending, so in both of the other two cases we obtain that the support of $x$ contains the range of an infinite ascending path.
\end{proof}

We also need the following simple observations.

\begin{lemma}
\label{lem:StringLeeway}
Let $x$ be a preroad and $p : \Z \to \Z^2$ an ascending path whose range is contained in the support of $x$. If the distance from the range of $p$ to every cell in the support of $x$ is bounded, then $x$ is a road. 
\end{lemma}

\begin{proof}
Suppose that $m$ is such that $x_v \neq 0 \implies \exists n: |p(n) - \bol v| \leq m$. Then let $a = |[-m,m]^2|$ and construct a path $q$ as follows: Let $n \in \Z$ and enumerate all cells in the support of $x$ at distance at most $m$ from $p(n)$ as $\bol v_0, \bol v_1, \ldots, \bol v_{\ell-1}$, where $\ell \leq a$. Define $q(na + j) = \bol v_j$ for all $j \in [0, \ell-1]$ and $q(na + j) = \bol v_1$ for $j \in [\ell, a-1]$. Clearly the support of $x$ is precisely the range of $q$, and the path $q$ is clearly ascending since $p$ is (though possibly with different ascension constant).
\end{proof}

\begin{lemma}
\label{lem:RoadIsHighway}
Let $X$ be almost minimal and $x \in X$ be a road. Then $x$ is a highway.
\end{lemma}



\begin{proof}
Shift $x$ suitably to obtain an ascending path $p \in \Path^2$ with $p(0) = \bol 0$ whose range is equal to the support of $x$. Let $p'$ be the corresponding colored path $p'(i) = (p(i), x_{p(i)})$. Take a uniformly recurrent path $q$ in the orbit closure of $p'$ in $\Path^2_\Sigma$. Then $q$ is obtained by taking a limit of some sequence of translates of $p'$. Define $y \in \Sigma^{\Z^2}$ by $y_{\bol v} = a$ if $q(n) = (\bol v, a)$ for some $n \in \Z$, $a \in \Sigma$, and $y_{\bol v} = 0$ otherwise. One can show that $y$ is well-defined and is in the orbit closure of $x$ by following the translates of $p'$ that tend to $q$ (we need the fact that $p$ is ascending).

Since $q$ is uniformly recurrent, $y$ is a highway. By almost minimality $y$ contains the same patterns as $x$. It follows that also $x$ must be a highway: every finite pattern in $x$ is traced by a subpath of $q$ (up to translation), and thus the support of $x$ is the range of a colored path in the orbit closure of $q$.
\end{proof}

Note that the paths $p, p'$ need not be uniformly recurrent. It is always possible to parametrize the range of a uniformly recurrent path in a non-recurrent way, for example for a uniformly recurrent injective $p \in \Path^2$ consider $q(n) = p(n - 1)$ for $n \geq 1$, $q(n) = p(n)$ otherwise. The lemma states that there is always automatically also a uniformly recurrent parametrization (for $q$, one such parametrization is $p$).

\section{The main theorem}
\label{sec:Main}

In this section, we prove our main theorem, the characterization of almost minimal trace sparse subshifts. Before this, as a warm-up, we give a characterization of one-dimensional almost minimal subshifts.

\begin{theorem}
\label{thm:OneDimensionalMainProof}
Let a subshift $X \subset \Sigma^\Z$ be almost minimal and nontrivial. Then exactly one of the following holds:
\begin{itemize}
\item $X$ is the orbit closure of a finite point, or
\item $X$ is the orbit closure of a \froctal{}.
\end{itemize}
\end{theorem}

\begin{proof}
Fix a nonzero point $x \in X$ with support $S \subset \Z$. If the language of $x$ does not contain arbitrarily long words of the form $0^n$, then the orbit closure does not contain $0^\Z$, contradicting almost minimality. It follows that for every $r \in \N$, there exists $n \in \N$ such that for all $i \in \Z$ the $r$-connected component of $i$ in $S$ is of diameter at most $n$. Collecting the subwords of $x$ corresponding to these connected components, we obtain a set of finite $r$-\blabs{}, which must be finite, since the \blabs{} are words of length at most $n+2r$.

Now, we start building a \froctal{} structure for $x$. Let $r_1 \in \N$ be arbitrary, and let $B_1$ be the finite collection of \blabs{} obtained across all $i \in S$ as in the previous paragraph. If $x$ is not a finite point, then $|i - j|$ can be arbitrarily large for $i, j \in S$. From this, it follows that for large enough $r$, there is an $r$-\blab{} $b$ whose nonzero support is strictly larger than that of any of the \blabs{} in $B_1$. By picking $r$ large enough, this is true for all $r$-\blabs{}, because $b$ appears at a bounded distance from every nonzero symbol by almost minimality. By increasing $r$ yet more, every \blab{} in $B_1$ appears as a subpattern of every $r$-\blab{}. Pick $r_2 \gg r_1$ with these properties, and define $B_2$ as the set of $r_2$-\blabs{}.

Continuing inductively, we obtain that $x$ is a \froctal{}.
\end{proof}


\begin{lemma}
Let $X \subset \Sigma^{\Z^2}$ be the orbit closure of a nontrivial finite configuration, a highway or a \froctal{}. Then $X$ is almost minimal.
\end{lemma}

\begin{proof}
The case of finite configurations is trivial.


For highways $x$, let $P$ be a nonzero pattern in $x$ with a connected domain. Pick an arbitrary nonzero cell $\bol v$ in $P$ and follow the colored path $p$ giving the support of $x$ forward and backward from the cell $\bol v$ for $t$ steps, to obtain some finite subpath $q$ of $p$ of length $2t+1$ in the support of $x$. Because the support of $x$ is the range of $p$ and $p$ is ascending, there is a function $f$ of the size of the domain $D(P)$ of $P$ such that if $t > f(|D(P)|)$ then $q$ visits all cells of $P$ and $p$ can never again visit the domain $D(P)$. Since $p$ is uniformly recurrent, the subpath $q$ is traced every $m$ steps for some $m$, so in particular $P$ occurs at a bounded distance from every cell in the support of $x$, and thus the same is true for its orbit closure $X$. 


For \froctals{}, let $B_i$ be the sets of $r_i$-\blabs{} as in the definition. If $x \in X$ and $P$ is a nonzero pattern in $x$, then let $i$ be such that $r_i$ is larger than the diameter of the domain of $P$. Then $P$ is contained in a single $r_i$-\blab{} $B \in B_i$ in $x$. Now, let $y$ be any point in $X$ and $\bol v$ any nonzero cell in $y$, i.e. $y_{\bol v} \neq 0$. Since $y$ is in the orbit closure of $x$, it is also a limit of \blabs{}, and thus $\bol v$ must be contained in the support of a translate of a $B_{i+1}$-\blab{}. This \blab{} contains a copy of $B$, and thus the pattern $P$.
\end{proof}

Finite configurations and highways (even roads) of course always have essential trace sparseness one. \Froctals{} can have infinite essential trace sparseness, and their essential trace sparseness can be any finite number even in the trace sparse case, see Section~\ref{sec:Examples}.

The following lemma is easy to prove.

\begin{lemma}
Let $X$ be the orbit closure of a finite configuration (highway or \froctal{}, respectively). Then every nonzero point in $X$ is a finite configuration (highway or \blab{} fractal, respectively).
\end{lemma}

\begin{theorem}
\label{thm:WhatminimalCases}
Let $X$ be an almost minimal and trace sparse $\Z^2$-subshift. Then exactly one of the following holds:
\begin{itemize}
\item $X$ is the orbit closure of a finite configuration,
\item $X$ is the orbit closure of a highway, or
\item $X$ is the orbit closure of a \froctal{}.
\end{itemize}
\end{theorem}

\begin{proof}
By the previous lemma, $X \subset \Sigma^{\Z^2}$ cannot be the orbit closure of two of these types of configurations. An almost minimal trace sparse subshift is uniformly trace sparse by Lemma~\ref{lem:EventuallyZeroRight}. Suppose the trace is $k$-sparse. 
By the assumption of almost minimality, we only have to find a point of one of these types in $X$. We begin with exactly the same argumentation as in Theorem~\ref{thm:OneDimensionalMainProof}. Take a point $x \in X$ and start building a \froctal{} structure for it by increasing $r$, collecting the $r$-\blabs{} that occur in $x$, and iterating.

If we can continue this for an infinite number of steps, and at each step we obtain finitely many distinct \blabs{} that each contain at least two \blabs{} of the previous size, we get in the limit that $X$ is the orbit-closure of a \froctal{} (the fact that all $B_{i+1}$-\blabs{} contain all $B_i$-\blabs{} is automatic if we increase $r$ fast enough, again due to almost minimality). This process can stop for two reasons. First, it can stop because $B_i$ is a singleton, and all $r_{i+1} > r_i$ only give padded versions of the \blab{} in $B_i$. In this case, $x$ is a finite configuration. In the other case, we have either infinitely many $r$-\blabs{} or a single infinite $r$-\blab{} in $x$, for some $r$. 
In particular, supports of configurations of $X$ contain arbitrarily large $r$-components for some $r$. By Lemma~\ref{lem:MPC}, $X$ then contains a preroad configuration. In the rest of the proof we show that then $X$ is the orbit closure of a highway.

Define the set $Y_m \subset (\Sigma^{\Z^2})^m$ as the set of $m$-tuples of configurations such that for any $n$, the $n \times n$ central patterns of those configurations can be found in the same $n$ rows of a configuration of $X$, separated by at least distance $n$ (but allowing said configuration to contain more nonzero symbols on those rows, and anything outside these $m$ many $(n \times n)$-rectangles). More precisely,
\begin{align*}
(x^1, \ldots, x^m) \in Y_m \iff &\forall n: \exists x \in X: \exists n_1, \ldots, n_m \in \Z: \forall i \neq j: \\
&|n_i - n_j| \geq 2n \wedge \sigma^{(n_i,0)}(x)_{[0,n-1]^2} = (x^i)_{[0,n-1]^2}
\end{align*}
Note that every $Y_m$ is closed under simultaneous translation in all coordinates $\sigma^{\bol v}(x^1, \ldots, x^m) = (\sigma^{\bol v}(x^1), \ldots, \sigma^{\bol v}(x^m))$. It is also closed under horizontal translation of any individual component: $(x^1, \ldots, x^i, \ldots x^m) \in Y_m \implies (x^1, \ldots, \sigma^{(j,0)}(x^i), \ldots x^m) \in Y_m$.

What we have shown is that in $Y_1 = X$, we find an $r$-preroad configuration for some $r$. Let $m$ be maximal such that in $Y_m$ we find an $m$-tuple $(x^1, \ldots, x^m)$ of $r$-preroad configurations.\footnote{The $r$ must be the same, but the ascension constants of the paths can a priori be different.} Naturally, there is an upper bound on $m$ as a function of $r$ and $k$, since all paths go through one of the rows $0, 1, \ldots, r-1$. Now, for each $i \in [1,m]$ fix an ascending path $p_i : \Z \to \Z^2$ whose image is contained in the support of $x^i$. If every nonzero symbol in $x^i$ is at most a bounded distance away from the range of $p$, then $x^i$ is a path configuration by Lemma~\ref{lem:StringLeeway}, and we are done.

Suppose then that $x^i$ is not a path configuration for some $i$, and let $x = x^i$ and $p = p_i$. We claim that $m$ is not maximal. Namely, let $\ell \in \N$, take a finite subpath $q$ of length $\ell$ in $p$, and let $Q$ be the corresponding pattern. Take $h$ such that a translate of $Q$ appears in the $h$-ball around every cell in the support of every configuration in $X$ (by almost minimality). Now, take a cell $\bol v$ in the support of $x$ that is at distance at least $\ell + r\ell + h$ from the central path $p$. Then a translate of $Q$ occurs in $x$ in the $h$-ball around $\bol v$, and thus every coordinate of the translate of $Q$ is at distance at least $\ell$ from the path $p$.

Now, we are in the following situation: The pattern $Q$ appears in $x$ in some position $(a_{m+1}, b)$, and on a nearby row in $x$, there is a nonzero coordinate $(a_i, b) + \bol r^i$ that lies on the central path, where $\bol r^i \in [-r,r]^2$ (because the path is ascending, and thus visits a syndetic set of rows), and $|a_i - a_{m+1}| \geq \ell$. Similarly, the paths $p_j$ in the other points $x^j$ visit some coordinates $(a_j, b) + \bol r^j$ where $\bol r^j \in [-r,r]^2$.

By translating $(x^1, \ldots, x^m)$ by $\sigma^{(a_i, b) + \bol r^i}$, and taking $\ell$ large enough, for any $\ell' \in \N$ we may assume that a translate of $Q$ occurs in the support of $x$ at distance at least $\ell'$ from the origin, visiting some cell $\bol v^{\ell'}$ on one of the rows $0,1,\ldots, r-1$ and extending $\ell'$ steps both below and above row zero, and the central path $p$ goes through the origin of $x$. By applying the horizontal translations $\sigma^{(a_j,0)}$ to the $j$th component for all $j \neq i$, we may assume the path $p_j$ in $x^j$, for all $j \in [1,m]$, visits one of the coordinates in $[-r,r]^2$. 

It follows that for every $\ell'$ we can find a tuple $(y^1, \ldots, y^m) \in Y_m$ where each $y^j$ contains a path that visits one of the cells in $[-r,r]^2$, and in $y^i$, a finite ascending path extends $\ell'$ rows upward and downward from some cell $\bol v^{\ell'}$ arbitrarily far from the origin, but at vertical distance at most $r$ from row zero. As $\ell' \rightarrow \infty$, we then find in the limit a tuple $(z^1, \ldots, z^{m+1}) \in Y_{m+1}$ where all of the $z^i$ are preroads. This is a contradiction with the maximality of $m$, and thus $x$ must have been a road. By Lemma~\ref{lem:RoadIsHighway}, $x$ is a highway.
\end{proof}

\section{Corollaries}

In this section, we give a list of corollaries of Theorem~\ref{thm:WhatminimalCases}. The most general extraction theorem obtained from Theorem~\ref{thm:WhatminimalCases} and Lemma~\ref{lem:Extract} is the following, which we mentioned in the introduction.


\begin{theorem}
\label{thm:MainFindingProof}
Let $X$ be any two-dimensional subshift whose trace $Y$ satisfies one of the following:
\begin{itemize}
\item $Y$ is sparse,
\item $Y$ is countable and the only periodic point in $Y$ is $0^\Z$.
\item every point in $Y$ is eventually zero to the right, or
\end{itemize}
Then $X$ contains a finite point, a highway, or a (trace sparse) \froctal{}.
\end{theorem}

\begin{proof}
In any of the three cases, Lemma~\ref{lem:Extract} applies and we find an almost minimal trace sparse subsystem. The result follows from Theorem~\ref{thm:WhatminimalCases}.
\end{proof}

\subsection{SFTs, sofics and zero-gluing}


\begin{proposition}
Let $X$ be a nontrivial trace sparse $\Z^2$-subshift such that for some~$r$, in the supports of all configurations of $X$ all $r$-components are infinite. Then $X$ contains a highway.
\end{proposition}

\begin{proof}
If $X$ is trace sparse, then it contains an almost minimal trace sparse nontrivial subsystem by Lemma~\ref{lem:RowsparseContainsAlmostMinimal}. Let $Y$ be an almost minimal subshift of $X$. Then $Y$ is the orbit closure of a finite configuration, a highway or a \froctal{} by Theorem~\ref{thm:WhatminimalCases}. Only highways have an infinite $r$-component for some~$r$.
\end{proof}

\begin{proposition}
\label{prop:RowsparseZeroGluing}
No trace sparse nontrivial $\Z^2$-subshift $X$ allows zero-gluing.
\end{proposition}

\begin{proof}
As above, in all three cases of the classification theorem, applied to an almost minimal subsystem $Y$ of $X$, we can use zero-gluing to show that the trace is not sparse: For finite and highway configurations, simply glue them to their horizontal translates. For \froctals{}, take an $r_i$-\blab{} for $r_i$ larger than the gluing radius, and the \blab{} must extend to a finite point because we can legally zero-glue it to the all-zero point.
\end{proof}

\begin{theorem}
\label{thm:PavlovSchraudnerProof}
No nontrivial two-dimensional SFT is trace sparse.
\end{theorem}

\begin{proof}
An SFT clearly allows zero-gluing.
\end{proof}

\begin{proposition}
Let $X$ be a two-dimensional sofic shift with SFT cover $Y$ such that the only preimage of $0^{\Z^2}$ is $0^{\Z^2}$. Then $X$ does not have a sparse trace.
\end{proposition}

\begin{proof}
Since the only preimage of a large ball of zeroes in the covering map is necessarily a large ball of zeroes, such a sofic shift allows zero-gluing.
\end{proof}

General sofic shifts can have sparse traces. See Section~\ref{sec:Examples}.

A $\Z$-subshift has \emph{universal period $n$} ($n \in \N, n \geq 1$) if there exists a bound $M \in \N$ such that for every point $x \in X$ there is a finite set $F_x \subset \Z$ of coordinates with $|F_x| \leq M$ and a periodic point $y \in X$ with $\sigma^n(y) = y$ such that $x|_{\Z \setminus F_x} = y|_{\Z \setminus F_x}$. We say it has \emph{universal period} if it has universal period $n$ for some $n$. (see Definition~4.3 in \cite{PaSc15}) 

The following is Theorem~6.4 in \cite{PaSc15}, and it is the non-implementability part of their characterization of zero-entropy sofic traces of two-dimensional subshifts of finite type.


\begin{theorem}
\label{thm:FullPavlovSchraudnerProof}
If a $\Z$-subshift $Y$ has a universal period and is not a finite union of periodic points, then it is not the trace of any $\Z^2$-SFT $X$.
\end{theorem}



\begin{proof}
Suppose $Y$ is the trace of a $\Z^2$-SFT $X$. Consider the \emph{blocking} of the subshift $X$, defined as the subshift with space $X$ and $\Z^2$-dynamics $\gamma^{(a,b)}(x) = \sigma^{(na,b)}(x)$. Picking $n$ to be the universal period of $Y$, this turns $Y$ into a finite union of sparse subshifts $Y_1, ..., Y_m$. Since there is a uniform bound on the number of coordinates where a point $y \in Y$ can differ from a unary point, we may assume that the zero symbol of $Y_i$ cannot appear in $Y_j$ for $j \neq i$.

Thus, we may assume that $X$ is an SFT whose trace is a finite union of sparse subshifts $Y_1, ..., Y_m$ with the above disjointness property on zero-symbols. Let $Z$ be a factor of $X$ obtained by mapping all the distinct zero symbols on every row of $x \in X$ to the same symbol $0$. Note that no non-unary row is mapped to $0^\Z$. Now, $Z$ is clearly trace sparse and nontrivial, so by Lemma~\ref{lem:RowsparseContainsAlmostMinimal}, it contains an almost minimal subsystem $W$.


By Theorem~\ref{thm:WhatminimalCases}, $W$ contains either a finite configuration, a highway configuration or a \froctal{}. If $z \in Z$ is finite or a highway, then clearly the zero-gluing $\sum_{i \in \Z} \sigma^{(ik,0)}(z)$ is in $X$ for some $k$, since we can simply glue the preimages of $z$ together as on each row they use the same preimage of the zero-symbol on the left and right side of the support of $z$. This contradicts the assumption that the original subshift had a universal period.

If $z$ is a \froctal{}, then pick any covering point $x \in X$ and some $r$-\blab{} $b$ in $x$ with $r$ larger than the maximal size of a forbidden pattern of the SFT, and replace every cell outside $b$ that is not one of the zero symbols with the zero-symbol used on that row of $x$. 
This configuration is legal, and all but finitely many of its rows are periodic. As in the previous paragraph, we can glue together a legal configuration contradicting the universal period.
\end{proof}

More generally, we obtain that there are singly periodic points in all SFTs whose traces are countable sofic shifts. We say a $\Z$-subshift $Y$ is \emph{bounded} if its language is bounded, where a bounded language is any sublanguage of $w_1^* w_2^* \cdots w_{\ell}^*$ for words $w_i$. In \cite{Sa14} it is shown that all countable sofic shifts are bounded. A \emph{singly periodic} configuration is a configuration with a nontrivial period whose orbit in the shift action is infinite. The following proposition should be constrasted with \cite[Theorem~3.11]{BaDuJe08}, where it is shown that all countably infinite two-dimensional SFTs contain a singly periodic point.


A configuration $x \in \Sigma^{\Z^2}$ is \emph{(totally) periodic} if $\{\bol u \in \Z^2 \;|\; \sigma^{\bol u}(x) = x\}$ spans $\R^2$ as an $\R$-vector space. We say a point $x \in \Sigma^{\Z^2}$ is \emph{asymptotic to a periodic point} if there exists a totally periodic $y \in \Sigma^{\Z^2}$ such that $x_{\bol v} = y_{\bol v}$ for all but finitely many $\bol v$. We say $x$ is \emph{strictly asymptotic to a periodic point} if it is asymptotic to a periodic point but is not periodic.
 The following lemma is easy to prove.

\begin{lemma}
\label{lem:EPUncountablePS}
Suppose $X \subset \Sigma^{\Z^2}$ is an SFT containing a point that is strictly asymptotic to a periodic point. Then the trace of $X$ is uncountable.
\end{lemma}


\begin{proposition}
\label{prop:TraceBoundedSinglyPeriodic}
If $X \subset \Sigma^{\Z^2}$ is an SFT whose trace is a bounded infinite subshift, then it contains a singly periodic point. 
\end{proposition}

\begin{proof}
Suppose $X$ is an SFT whose trace is a bounded infinite subshift, and suppose it contains no singly periodic point. First, we show that $X$ contains only finitely many points which have a horizontal period: The trace of $X$ contains only finitely many periodic points. Considering the periodic rows as a finite alphabet, the set of points with all rows periodic is an SFT under the vertical action. If this one-dimensional SFT were infinite, it would contain an eventually periodic aperiodic point, giving a singly periodic point in $X$. Thus this one-dimensional SFT is finite, implying that indeed there are only finitely many points with a horizontal period.

Let now $n$ be such that every point with some horizontal period has horizontal and vertical period dividing $n$, and suppose that $X$ is defined by forbidden patterns of size at most $n$-by-$n$. Suppose $0 \notin \Sigma$ and consider the subshift $Z$ obtained from $X$ as the image of the factor map $f : X \to (\Sigma \cup \{0\})^{\Z^2}$ defined by $f(x)_{\bol v} = 0$ if $x_{\bol v + (a, b)} = x_{\bol v + (a + n, b)} = x_{\bol v + (a, b + n)} = x_{\bol v + (a + n, b + n)}$ for all $0 \leq a, b < n$, and $f(x)_{\bol v} = x_{\bol v}$ otherwise. The trace of $Z$ is then sparse. By the assumption on the trace of $X$, the trace of $Z$ is not $\{0^\Z\}$. We get that $Z$ contains a nontrivial finite point, a highway or a \froctal{}.

If $Z$ contains a highway $z$, then from a covering configuration in $X$ of $z$, we obtain a singly periodic point 
in $X$ by the pigeonhole principle, since the preimage of a finite point in the trace of $Z$ is clearly an eventually periodic point in the trace of $X$, and there are only finitely many points with a given period $n$ for the left and right tails, and with a given central pattern length.


For a nontrivial finite configuration we can find a covering point of the following form: for some $m$, $x_{\bol v} = x_{\bol v + (n,0)} = x_{\bol v + (0,n)} = x_{\bol v + (n,n)}$ for all $|\bol v| \geq m$, but $x$ is not periodic. But then $x$ is strictly asymptotic to a periodic point, a contradiction with Lemma~\ref{lem:EPUncountablePS}. 



Now, let $z$ be a \froctal{} in $Z$ and let $x \in X$ be a covering configuration. We claim that it is enough to show that around any $1$-component of the support of $z$ (note that for the definition of $1$-components we use the \emph{king grid}, that is, the graph whose path metric is the $\ell^\infty$-metric $d$), we can draw a path along zeroes around it with movements in cardinal directions only. More precisely, that we can find a simple cycle $p$ in the \emph{standard grid} $(\Z^2, \{(u,v) \;|\; |u_1 - v_1| + |u_2 - v_2| = 1\})$ whose interior (as a Jordan curve) contains the component, and $z_{p(i)} = 0$ for all $i$.

To see that this suffices, once we have such a path, by the assumption on the factor map and on $n$, throughout this cycle, we see a horizontal and vertical period of $n$ in $x$, meaning that the unique point $y$ in $X$ with horizontal and vertical period $n$ that agrees with $x$ in some $2n$-by-$2n$ square with bottom left corner on the path $p$ agrees with $x$ in all the $2n$-by-$2n$ blocks on the path $p$. If $b$ is the \blab{} corresponding to the component, then by the assumption that its support is contained in the interior of $p$, we can ignore the configuration outside $b$, and continue the period to obtain a point strictly asymptotic to the periodic point $y$, which is a contradiction by Lemma~\ref{lem:EPUncountablePS}.

For the existence of the path, consider the family of counterclockwise cycles in the standard grid, whose interior contains the support of $b$. One such Jordan curve exists because $b$ is finite, and thus there exists such a curve $p$ of minimal area. A simple case analysis shows that every cell $p(i)$ is in the king grid neighborhood of the support of $b$, or we could decrease the area of the curve by cycling around more internal nodes.\footnote{For a more rigorous proof, one needs that the two sides of a curve are really on different sides, so that always we have an internal side to decrease in area (and that we have a notion of area at all). The polygon version of the Jordan curve theorem suffices for this \cite[Lemma~1]{Tv80}.}
\end{proof}

Of course this implies also that two-dimensional SFTs with bounded trace contain doubly periodic points, but this has an easier direct proof by passing to a minimal subsystem.

\subsection{Subshifts with deterministic or expansive directions}

SFTness can often be replaced by determinism. The main thing we need is that if $\bol v$ is a deterministic direction, then a half-plane in direction\footnote{Recall that by this we mean that the boundary is perpendicular to $\bol v$ -- we think of the half-plane as moving in direction $\bol v$ and eating up the configuration.} $\bol v$ containing only $0$s (or other doubly periodic content), can only be continued periodically. If $X$ contains a configuration contradicting this, we say \emph{something appears from nothing}. The \emph{vertical deterministic directions} are $\{r(0,1), r(0,-1) \;|\; r \in \R_+\}$.


\begin{proposition}
\label{prop:UpwardDeterminism}
If $X \subset \Sigma^{\Z^2}$ is a subshift with a vertical deterministic direction whose trace is a bounded infinite subshift, then it contains a singly periodic point. 
\end{proposition}

\begin{proof}
Suppose $X$ is deterministic upward.

The first two paragraphs of the proof of the previous proposition go through with minor modifications, in the first observing that a one-dimensional subshift with a deterministic direction is an SFT, and in the second picking $n$ to be the radius of the local rule of upward determinism instead of the SFT window. Let $Z$ be as in that proof.

Let now $x \in X$ be a covering configuration of a highway configuration $z \in Z$. Determinism means that a lower half-plane containing half of the path will have a unique extension to a configuration of $X$. Since the determinism is given by a local rule, in fact a finite number $\ell$ of rows already determines the next row upward. Since every row of $x$ is eventually periodic with uniformly bounded preperiod (by the definition of a highway, and the definition of the factor $Z$), there are, up to shifting, only finitely many $\ell$-tuples of consecutive rows. It follows from the pigeonhole principle that the direction of the path is rational, and thus $x$ is singly periodic.

There cannot be any finite configurations in $Z$ because otherwise something appears from nothing in $X$. Similar reasoning applies to \froctals{}: in the orbit closure of a \froctal{}, for any direction $\bol v \in \R^2 \setminus \{\bol 0\}$ (in particular for $\bol v = (0,1)$) one can find a half-plane in direction $\bol v$ of all zeroes whose boundary contains a nonzero symbol, and again something appears from nothing in $X$.
\end{proof}

When the deterministic direction is not vertical, $X$ need not contain a singly periodic point. See Example~\ref{ex:BoundedDeterministicSubshift} in Section~\ref{sec:Examples}. However, we will show next that in the trace sparse case, any direction of determinism implies the existence of a singly periodic point. A much studied concept in multidimensional dynamics are the directions of expansivity \cite{BoLi97}, a concept similar to determinism. We say $\bol v \in \R^2 \setminus \{\bol 0\}$ is an \emph{expansive direction}\footnote{Of course, a more natural way to define directional expansivity is to talk about expansive subspaces of $\R^2$, but for the purpose of our discussion, we find directions notationally easier.} if, defining $L_{\bol v,r} = \R \bol v + B_r(0) \subset \R^2$, we have
\[ \exists \epsilon > 0: \exists r > 0: (\forall \bol u \in L_{\bol v,r} \cap \Z^2: d(\sigma^{\bol u}(x), \sigma^{\bol u}(y)) < \epsilon) \implies x = y. \]
Determinism and expansivity are connected by the following well-known lemma.

\begin{lemma}
Let $X$ be a two-dimensional subshift. Then $\bol v$ is an expansive direction for $X$ if and only if both of the directions orthogonal to $\bol v$ are deterministic.
\end{lemma}

\begin{proof}
Let $D$ be the set of deterministic directions, and $E$ the set of expansive ones (which is of course symmetric). We will show $\bol v \notin D \implies (\bol v^T \notin E \wedge -\bol v^T \notin E)$ and $\bol v \notin E \implies (\bol v^T \notin D \vee -\bol v^T \notin D)$, from which the lemma follows easily.

If $\bol v$ is a direction of nondeterminism, then we have two points $x,y$ that agree on a $\bol v$-directional half-plane, but not everywhere, and then thick strips around the boundary of this half-plane imply the directions orthogonal to $\bol v$ are non-expansive. Conversely, if $\bol v$ is a direction of nonexpansivity, then we have arbitrarily thick strips $L_{\bol v,r}$ that extend in two ways in at least one of the directions. Translating this difference to the origin, in the limit as $r \rightarrow \infty$ we obtain two points that agree on a half-plane in one of the two directions orthogonal to $\bol v$, but not everywhere.
\end{proof}


In \cite{BoLi97}, $\Z^2$-subshifts with almost any\footnote{Up to natural topological restrictions.} set of expansive directions are constructed, the only exception being a single irrational expansive direction, which was left open. Since a rational line in direction $\bol v$ will have $\bol v$ as the unique nonexpansive direction, the obvious way to try to solve this is a single irrational line. The next result, Proposition~\ref{prop:DeterministicContainsSinglyPeriodic}, shows that such an idea cannot work. 
See \cite{Ho11} for a correct implementation of a unique nonexpansive irrational direction.

We need the following well-known fact, which can be shown analogously as the closedness of nonexpansive directions in \cite{BoLi97}. 

\begin{lemma}
Let $X \subset \Sigma^{\Z^2}$ be a subshift. Then the set of nondeterministic directions (scaled onto the unit circle) is closed in the unit circle.
\end{lemma}

\begin{lemma}
If a subshift $X \subset \Sigma^{\Z^2}$ is nontrivial, almost minimal and trace sparse, and has a direction of determinism, then it is the orbit closure of a singly periodic point, and its set of nondeterministic directions is $\{r\bol v', -r\bol v' \;|\; r \in \R_+\}$ for some $\bol v' \in \Z^2$.
\end{lemma}

\begin{proof}
Again since $X$ has a deterministic direction, it has a rational deterministic direction $\bol v$ by the previous lemma. It cannot contain a finite configuration or a \froctal{}, and thus must consist of $r$-highways for some fixed $r$. Let $H$ be a half-plane in direction $\bol v$ whose boundary is at the origin. Now, consider a minimal subsystem $Y$ of $\MPC_r(X)$. To each path $p : \Z \to \Z^2$ corresponding to a point in $Y$, associate the sequence of dot products $n \mapsto \langle p(n), \bol v \rangle$. This sequence measures how far the path is from the boundary of $H$.

Since the paths are uniformly recurrent, it follows from Theorem~\ref{thm:PathCharacterizationGeneral} that either some path enters some strip around the boundary of $H$ infinitely many times, or every path deviates from the boundary with some uniformly lower bounded velocity (in the sense that the sequence of dot products is an ascending path).

First, suppose some path visits some strip around the boundary of $H$ infinitely many times. Then it is easy to find a configuration where something appears from nothing, by looking at the times when the path is maximally far from the boundary in the direction $-\bol v$. Thus, this case is impossible.

Suppose then that $n \mapsto \langle p(n), \bol v \rangle$ is a uniformly ascending for every path $p$ corresponding to a point in $\MPC_r(X)$, that is, there exist $t, m$ such that if $p$ is any such path and $p(t)$ is in the inner $m$-border of $H$, i.e.\ $p(t) \in A = (\partial H + B_m(\bol 0)) \cap H$ then for all $t'' \notin [t'-t,t'+t]$, $p(t) \notin A$.

Now, observe that the border of any translate of $H$ is periodic, in the sense that for any $\bol u \in \R^2$ the characteristic function of $(\bol u + \R \bol v^T + B_{m/2}(\bol 0)) \cap \Z^2$ is in the orbit of one of finitely many singly periodic configurations in $\{0,1\}^{\Z^2}$, and uniform ascension of paths implies that there are only finitely many distinct configurations that can be seen on the border of $H$. From determinism in direction $\bol v$ it follows that every path is periodic, and almost minimality then implies that the subshift is the orbit closure of a singly periodic point. Pick $\bol v' = \bol v^T$.
\end{proof}

By extracting an almost minimal subsystem and applying the previous lemma, we obtain the following.

\begin{proposition}
\label{prop:DeterministicContainsSinglyPeriodic}
If a subshift $X \subset \Sigma^{\Z^2}$ is nontrivial and trace sparse and has a direction of determinism, then it contains a singly periodic point, and thus has a rational nondeterministic direction. 
\end{proposition}

\subsection{Cellular automata}

Surjective cellular automata are essentially just subshifts with a rational deterministic direction, and thus it is clear that we obtain corollaries for cellular automata.

The main technical tool used in \cite{SaTo12cn} is the Starfleet Lemma for cellular automata on countable sofic shifts, which states that, starting from any configuration, a `fleet of spaceships' appears infinitely often, in the same configuration. This result shows that all cellular automata on countable sofic shifts at least occasionally behave like counter machines. We do not state this result precisely, but we reproduce one of the main corollaries obtained from it in \cite{SaTo12cn}, namely the decidability of the nilpotency problem.


More precisely, let $X \subset \Sigma^\Z$ be a countable sofic shift. A cellular automaton $f : X \to X$ is \emph{nilpotent} if for some $n \in \N$, $f^n(x) = 0^\Z$ for all $x \in X$. A \emph{spaceship} for $f$ is a non-periodic eventually periodic configuration $x = {^\infty} u v w^\infty$ such that $f^n(x) = \sigma^m(x)$ for some $m \in \Z, n \geq 1$. If $u, w \in 0^*$, then $x$ is called a \emph{glider}.

\begin{theorem}
\label{thm:SaloTormaProof}
Let $X$ be a bounded one-dimensional subshift and $f : X \to X$ a cellular automaton. Then either there exists $k$ such that $f^k(X)$ is spatially periodic, or there exists a spaceship for $f$. 
In particular, nilpotency is decidable for cellular automata on countable sofic shifts.
\end{theorem}

\begin{proof}
If $f^k(X)$ is not finite for any $k$, then the horizontal trace of the spacetime subshift of $f$ is infinite. By Proposition~\ref{prop:UpwardDeterminism}, $X$ contains a singly periodic point. Such a configuration is precisely a spaceship for $f$. The decidability of nilpotency follows because both ``$\exists k: |f^k(X)| < \infty$'' and ``$x \mbox{ is a spaceship for } f$'' are semidecidable on sofic shifts, and in the former case we can easily check nilpotency, while the latter implies non-nilpotency.
\end{proof}


\begin{proposition}
\label{prop:TraceLimitSpaceshipProof}
Let $X \subset \Sigma^\Z$ be any one-dimensional subshift, and let $f : X \to X$ be a CA such that either $f$ is asymptotically nilpotent or the limit set of $f$ contains only configurations that are eventually zero to the right. Then $f$ is either nilpotent or has a glider.
\end{proposition}

\begin{proof}
If $f$ is asymptotically nilpotent, then consider its spacetime subshift $Y$ where $f$ is run to the right, so that rows are eventually zero to the right. By Lemma~\ref{lem:Extract} we can extract an almost minimal trace sparse subshift $Z$ from $Y$. Then $Z$ is a trace sparse almost minimal subshift with a deterministic direction, and thus contains a singly periodic point, which corresponds to a glider. The proof in the case of a limit set with all configurations eventually zero to the right is the same, but using a vertical deterministic direction.
\end{proof}


We can also obtain the following proposition. We note that this is stronger than the previous proposition in the sense that even on full shifts, there are cellular automata such that the closure of their asymptotic set is a proper subshift of their limit set.

\begin{proposition}
Let $X \subset \Sigma^\Z$ be any subshift, and let $f : X \to X$ be a CA such that the closure of the asymptotic set of $f$ contains only configurations that are eventually zero to the right. Then $f$ is either nilpotent or has a glider.
\end{proposition}

\begin{proof}
If the asymptotic set of $f$ contains only the all-zero point, then $f$ is asymptotically nilpotent, and the claim follows from the previous theorem. Otherwise, the closure $Z$ of the asymptotic set of $f$ is a subshift on which $f$ is surjective and that contains at least one nonzero point. Then the spacetime subshift of $f$ where the horizontal traces are restricted to be in $Z$ is a two-dimensional nonzero subshift where every row is eventually zero to the right, and the claim follows as in the previous theorem.
\end{proof}


In the case where $X$ is an SFT, we of course cannot have any such gliders in the nilpotent case, and we obtain the following theorem (which is also proved in \cite{GuRi10}).

\begin{theorem}
\label{thm:GuillonRichardProof}
Let $X$ be a one-dimensional SFT and $f : X \to X$ a cellular automaton. Then $f$ is nilpotent if one (equivalently all) of the following holds:
\begin{itemize}
\item $f$ has a sparse limit set,
\item $f$ has a sparse asymptotic set,
\item $f$ is asymptotically nilpotent, or
\item the (one- or two-sided) vertical trace subshift of $f$ is sparse.
\end{itemize}
\end{theorem}

Here, the two-sided vertical trace subshift of $f$ is the vertical trace of the spacetime subshift of $f$, and the one-sided vertical trace is $\{y \in \Sigma^\N \;|\; \exists x \in X: \forall i: y_i = f^i(x)_0\}$.

\subsection{Irrational sparseness}

Since Theorem~\ref{thm:Main} applies to any almost minimal subshift, it in particular applies to binary subshifts that are coding some property of points in another system. We show one application of this observation, namely that in Theorem~\ref{thm:Main} it is enough to assume sparseness in any (possibly irrational) direction.

If $\bol v \in \R^2 \setminus \{(0,0)\}$ is a direction and $X$ a subshift, we say $x \in X$ is a \emph{road in direction $\bol v$} if the support of $x$ is the range of a path $p : \Z \to \Z^2$ such that the sequence of dot products $n \mapsto \langle p(n), \bol v \rangle : \Z \to \R$ is a uniformly ascending path. It is a \emph{highway in direction $\bol v$} if the colored path $q$ defined by $q(i) = (p(i), x_{p(i)})$ is uniformly recurrent.

\begin{proposition}
Let $X$ be a $\Z^2$-subshift and let $L = \R \bol v$ be a line with arbitrary slope. Suppose that for some $r > 0$, the strip $\bol u + L + B_r(0)$ in $x$ contains finitely many nonzero symbols for every $\bol u \in \R^2$ and every $x \in X$. Then $X$ contains a finite point, a \froctal{} or a highway in direction $\bol v^T$.
\end{proposition}

\begin{proof}
Let $\bol v$ be the unit vector giving the slope of $L$. 
Let $R \subset \R^2$ be the closed square having one corner at $(0,0)$ and with sides $\bol v$ and $\bol v^T$. Now, to a point $x \in X$ whose nonzero support is $S \subset \Z^2$, associate the binary point $\psi(x) \in \{0,1\}^{\Z^2}$ defined by
\[ \psi(x)_{\bol u} = 0 \mbox{ iff } S \cap (\bol u_1 \bol v + \bol u_2 \bol v^T + R) \neq \emptyset. \]

Then $\psi(x)$ is a coding of a rotated version of $x$. Let $Y$ be the subshift generated by the points $\psi(x)$. Now, it follows from the assumption that the subshift $\psi(X)$ is trace sparse, and contains a point of one of the three types in Theorem~\ref{thm:WhatminimalCases}. It is easy to check that finite points must come from finite points through a rotation, and highways must come from roads in direction $d^T$, which we can turn into highways in direction $d^T$ by again passing to an almost minimal subshift. \Froctals{} need not a priori come from \froctals{} (since our definition of \froctal{} implies almost minimality). However, it is again clear that any almost minimal subsystem will be supported by a \froctal{}.
\end{proof}

\subsection{Topological full groups}


The \emph{topological full group} of a subshift $X$ is the group $G$ of homeomorphisms $h : X \to X$ such that for some \emph{cocycle} $c : G \times X \to \Z$, $c(h, x) = n \implies h \cdot x = \sigma^n(x)$ and $c$ is continuous in its right argument \cite{GiPuSk99}. Such a $c$ is a cocycle in the sense of cohomology, that is, $c(g \circ h, x) = c(h, x) + c(g, h(x))$.

\begin{lemma}
\label{lem:TopologicalFullGroup}
Let $g$ be an element of the topological full group $G$ of a subshift $X \subset \Sigma^\Z$ with cocycle $c$. If $g$ has infinite order, then there is a point $x \in X$ and $k \geq 1$ such that $n \mapsto c(g^{kn}, x)$ is increasing or decreasing.
\end{lemma}



\begin{proof}
Take the $\Z^2$-subshift $Y$ where rows are unary and columns are configurations of $X$. For $x \in X$, write $\bar x$ for the unique configuration of $Y$ with central column $x$. Add another layer on top of $Y$, where in cell $\bol v \in \Z^2$ of $\bar x \in Y$, we write the vector $(1, c(g, \sigma^{\bol v_2}(x)))$. To see that this gives a subshift, observe that by continuity $c(g, \sigma^{\bol v_2}(x))$ is determined by $\sigma^{\bol v_2}(x)|_F$ for a finite $F \subset \Z$, thus by $x|_{\bol v_2 + F}$ in a shift-invariant way, thus $(1, c(g, \sigma^{\bol v_2}(\bar x)))$ is determined by the values in $\bar x|_{\bol v + \{0\} \times F}$ in a shift-invariant way (since rows are constant). Call the resulting subshift $Z$.

(The relation between $Z$ and the path cover $\PC_r(Y)$ (with $r = \max |c(f, x)|$) is that $Z$ is defined similarly as $\PC_r(Y)$ except cycles are not forbidden, but additionally the symbol $\#$ is forbidden and furthermore the path is entirely determined by the local rule of the cocycle $c$.)

Now, the function $n \mapsto c(g^{kn}, x)$ tells us the vertical movement of a path in a configuration of $Z$, when we interpret it as a graph $G(Y)$ by following the vectors written on the second layer as in Section~\ref{sec:PathCovers}. This correspondence follows from the cocycle formula: If $\bar x_{\bol v} = (1, c(g, \sigma^{\bol v_2}(x))) = (1, n)$ then $g \cdot \sigma^{\bol v_2}(x) = \sigma^{\bol v_2 + n}(x)$. The arrow at $\bar x_{\bol v}$ points to the cell $\bol v + (1, n)$, which corresponds to the central coordinate of $\sigma^{\bol v_2 + n}(x)$.

If a path visits the same row twice, then it is periodic. Because $g$ is of infinite order, we must find paths that do not revisit any cell in arbitrarily many steps, meaning that the corresponding drawn path does not visit the same row twice.

Now, consider a subshift $W \subset Z \times \{0,1\}^{\Z^2}$ where we allow coloring exactly one of the paths in $Z$ with $1$s. Forbid, in $W$, revisits of the colored path to the same row. By the assumption of the previous paragraph $W$ still contains configurations with a colored path. Now, consider a factor of $W$ where the $Z$-component is removed. This subshift is trace sparse, and thus contains an almost minimal subshift, which must consist of a single uniformly ascending or descending path. By looking at the corresponding $Y$-configuration, the claim is proved.
\end{proof}

Of course the proof adapts easily also to topological full monoids, but we omit the definition. We obtain another proof of a result of \cite{BaKaSa16}. The \emph{torsion problem} of a recursively presented group $G$ is the problem of, given $g \in G$, determining whether the order of $g$ is finite. The topological full group of a full shift (more generally any $\Pi^0_1$ subshift) has an obvious recursive presentation by local rules, see \cite{BaKaSa16}.


\begin{theorem}
\label{thm:BarbieriKariSaloProof}
If $X$ is a full shift, then the torsion problem of the topological full group of $X$ is decidable.
\end{theorem}

\begin{proof}
Clearly the torsion problem is semidecidable. On the other hand, if $g$ has infinite order, then let $c$ be as above. By the previous lemma there is a point $x \in X$ and $k \geq 1$ such that $n \mapsto c(g^{kn}, x)$ is increasing or decreasing. By the pigeonhole principle, we can find such periodic $x$. By changing a single coordinate in the tail of $x$ that $g$ does not see, we find a spatially eventually periodic but non-periodic configuration that $g$ shifts periodically in one direction. This is clearly semidecidable.
\end{proof}

\subsection{Patterns in primes}

Unlike our main theorem, the one-dimensional version Theorem~\ref{thm:OneDimensionalMainProof} applies very generally (though of course also its conclusion is rather trivial). We show one application of it -- finding `infinite patterns' in primes. Let $N \subset \N$ be arbitrary, and $x_N \in \{0,1\}^\Z$ its characteristic function. We call
\[ X_N = \bigcap_m \overline{\{\sigma^n(x_N) \;|\; n \geq m\}} \]
the \emph{$N$-subshift}; it is always a subshift, and it is the set of two-way infinite configurations whose subwords appear arbitrarily late to the right in the point $x_N$. Setting $N$ to be the set $\mathcal{P}$ of prime numbers we obtain what we call the \emph{prime subshift} $X_{\mathcal{P}}$.

\begin{lemma}
Let $p_i$ be the $i$th prime number. Then 
\begin{equation}
\label{eq:Primes}
\liminf_{n \rightarrow \infty} p_{n+1}-p_n < \infty
\end{equation}
if and only if the prime subshift is not equal to the sunny-side-up subshift.
\end{lemma}


\begin{proof}
Since there are infinitely many primes, $X_{\mathcal{P}}$ is not the all-zero subshift by compactness. If \eqref{eq:Primes} holds, then there is some $k \in \N$ such that $p_{n_i+1} - p_{n_i} = k$ for some sequence $n_i \rightarrow \infty$. Let $y$ be a limit point of the sequence $\sigma^{p_{n_i}}(x)$ as $i \rightarrow \infty$. Then $y_{[0,k]} = 10^{k-1}1$ and $y \in X_{\mathcal{P}}$.

Suppose then that $X_{\mathcal{P}} \neq X_{\leq 1}$. Since $X_{\mathcal{P}}$ is not the all-zero subshift, some point in it contains a word with two ones, and thus a word of the form $10^{k-1}1$. But then $\liminf_{n \rightarrow \infty} p_{n+1}-p_n \leq k$.
\end{proof}

The fact that indeed $\liminf_{n \rightarrow \infty} p_{n+1}-p_n < \infty$ was shown in 2014 in a revolutionary paper by Zhang \cite{Zh14}. Since then, much more has been learned about the words appearing in $X_{\mathcal{P}}$, and in particular \cite{Ma15} shows that $10^n1 \sqsubset X_{\mathcal{P}}$ holds for some $n < 600$. For a discussion of the subshift point of view to the study of divisibility, see \cite{BaKaKuLe15} and references thereof.

The following is a corollary of Theorem A in \cite{BaKaKuLe15}, where it is shown that every \emph{$\mathcal{B}$-free subshift} (see their paper for the definition) is essentially minimal. Our definition of the prime subshift is slightly different from theirs, so we give the proof.

\begin{lemma}
The prime subshift is essentially minimal, and its only minimal subshift is $\{0^\Z\}$.
\end{lemma}

\begin{proof}
Let $n$ be arbitrary. Let $I = \{0,1,\ldots,n\}$ and let $\phi : I \to \mathcal{P}$ be an injection from $I$ into the prime numbers. Let $N = \prod_{i \in I} \phi(i)$. Let $k = \sum_{i \in I} a_i N / \phi(i)$ where the $a_i$ are chosen so that $k \equiv -i \bmod \phi(i)$ for all $i \in I$.
Then $k+i \equiv i-i \equiv 0 \bmod \phi(i)$ for all $i \in I$. It follows that
$x_{\mathcal{P}}$ contains the word $0^n$ for arbitrarily large $n$. If $(x_{\mathcal{P}})_{[k,k+n-1]} = 0^n$, then $k+i$ is divisible by $p_i$ for $i \in [0,n-1]$ for some prime $p_i$. It follows that $0^n$ occurs in every subword of $x_{\mathcal{P}}$ of length $\prod_{i = 0}^{n-1} p_i + 2n$.
\end{proof}


For a set $N \subset \N$, write $\pi_N(n) = |N \cap [0,n]|$, and define $\pi = \pi_{\mathcal{P}}$.
We note that the prime number theorem $\pi(n) \sim n/\log n$ or any other density-related statement is \emph{not} enough on its own, and it is rather the periodic structure of the zeros that is important -- by perturbing the primes by inserting intervals $[\ell,\ell+n]$ very sparsely, but for arbitrarily large $n$, one obtains a set $N \subset \N$ with $\pi_N \sim \pi$ such that $X_N$ is not essentially minimal. By alternating long intervals of zeroes and ones, it is also possible to build sets $N \subset \N$ such that $\pi_N \sim \pi$ and $X_N = \OC{...000111...} \cup \OC{...111000...}$.

\begin{proposition}
\label{prop:Primes}
The prime subshift contains either a finite point or a \froctal{}.
\end{proposition}

\begin{proof}
By the previous lemma, $X_{\mathcal{P}}$ is essentially minimal with $\{0^\Z\}$ the only minimal subsystem. By Lemma~\ref{lem:RowsparseContainsAlmostMinimal}, $X_{\mathcal{P}}$ contains an almost minimal subsystem, and the claim follows from Theorem~\ref{thm:OneDimensionalMainProof}.
\end{proof}

What is interesting about the proposition is mainly that we obtained it for essentially abstract reasons from only the Chinese remainder theorem and the infinitude of the primes, while proving that $X_{\mathcal{P}}$ contains a particular finite point or \froctal{} is quite difficult from first principles.

We do not know whether the prime subshift contains a \froctal{}. If it did, $X_{\mathcal{P}}$ would be uncountable. We note that at least the conclusion of Theorem~1.2 of \cite{Ma15} allows the prime subshift to be countable, as it is easy to show that the subshift
\[ X = \{ x \in \{0,1\}^\Z \;|\; (|x| \geq n \wedge x_i = x_j = 1) \implies (i = j \vee |i-j| \geq n) \}, \]
where $|x|$ is the number of nonzero symbols in $x$, satisfies its conclusion, in the sense that if you pick for a large enough set of integers, then a randomly chosen $m$-tuple taken from those integres will with high probability be the support of a legal configuration of $X$. (The same is true even if one further restricts $X$ to contain only admissible patterns.)

We remark that one \emph{can} show that $X_{\mathcal{P}}$ contains $X_{\leq 1}$ using a more sophisticated but standard number-theoretic tool, namely Dirichlet's theorem on arithmetic progressions.

\begin{proposition}
\label{prop:SunnyInPrimes}
The prime subshift contains the sunny-side-up subshift.
\end{proposition}


\begin{proof}
Let $n$ be arbitrary. Let $I = \{-n,\ldots,-1\} \cup \{1,\ldots,n\}$ and let $\phi : I \to \mathcal{P}$ be an injection from $I$ to the prime numbers all larger than $2n$. Let $N = \prod_{i \in I} \phi(i)$. Let $k = \sum_{i \in I} a_i N / \phi(i)$ where the $a_i$ are chosen so that $k \equiv i \bmod \phi(i)$ for all $i \in I$. Then $k$ and $N$ are coprime, so by Dirichlet's theorem on arithmetic progressions, there is a prime number $p = k + \ell N$ for some $\ell \in \N$. Then $p-i \equiv i-i \equiv 0 \bmod \phi(i)$ for $i \in I$. It follows that the characteristic sequence of the primes contains the word $0^n 1 0^n$ around the prime $p$. Taking a suitable limit, we obtain that $X_{\leq 1} \subset X_{\mathcal{P}}$.
\end{proof}

\section{Examples}
\label{sec:Examples}

In this section, we give some examples of almost minimal and sparse subshifts, to illustrate to what extent the various assumptions are needed in Theorem~\ref{thm:Main}, and in particular show what kind of behaviors can happen in sofic shifts. We construct these examples using existing constructions of sofic shifts as black-boxes. In particular we use the result of Mozes \cite{Mo89} that fixed-point subshifts of many substitutions (in particular rectangular deterministic ones) are sofic, and the result of various authors that traces of vertically constant sofics can be arbitrary computable subshifts \cite{Ho09,DuRoSh10,AuSa13}. We also give some examples of path spaces.

We note that many examples of subshifts (in particular sofic shifts and SFTs) with bounded traces are constructed in \cite{SaTo12c,SaTo13}, and bounded subshifts can be turned into sparse ones by blocking and factoring (see the proof of Theorem~\ref{thm:FullPavlovSchraudnerProof}). The proof of Lemma~\ref{lem:TopologicalFullGroup} shows that every element of a topological full group of a one-dimensional subshift gives an example of a trace sparse two-dimensional subshift.

\subsection{Subshifts}



The following is our first example of a trace sparse \froctal{}, and thus shows that in Theorem~\ref{thm:Main}, the case of \froctals{} must indeed be included. Less trivially, it also shows that while essential trace sparseness is always $1$ in the cases of finite configurations and path configurations, the essential trace sparseness of an almost minimal trace sparse subshift can in general be arbitrarily large and can match the sparseness constant.

\begin{example}
\label{ex:UnboundedRows}
For every $k$, there exists an almost minimal two-dimensional sofic subshift $X \subset \{0,1\}^{\Z^2}$ with $k$-sparse trace such that every nonzero configuration has essential trace sparseness $k$.
\end{example}

See Figure~\ref{fig:SparseBlobFractal} for an illustration of a typical pattern in the subshift $X$ we construct.

\begin{proof}
For $k = 1$, the two-dimensional sunny-side-up subshift has this property.

Let $k \geq 2$. We inductively construct $k$ patterns $P_{i,1}, \ldots, P_{i,k}$ of shape $[0,m_i-1]^2$ such that the bottom row of each contains exactly one $1$, which is at the bottom left corner: $(P_{i,j})_{0,0} = 1$. The patterns also satisfy that if $j \neq j'$ then the only row that is nonzero in both $P_{i,j}$ and $P_{i,j'}$ is row $0$ and that no row of $P_{i,j}$ contains more than $k$ nonzero symbols. Furthermore, each of the patterns $P_{i,j}$ has at least one row with $k$ nonzero symbols that are pairwise separated by at least distance $m_{i-1}$.

Pick as the patterns $P_{1,j}$ the following matrices (using $P_{0,j'} = 1$ for all $j'$), and observe that the assumptions are then satisfied for $i = 1$. Now, for $j \in \{1,\ldots,k\}$, build $P_{i+1,j}$ from the patterns $P_{i,1}, \ldots, P_{i,k}$ as follows. First, set $m_{i+1} = (k+1)m_i$, so that $[0,m_{i+1}-1]^2$ partitions into a $(k+1) \times (k+1)$ grid of translates of the squares $[0,m_i-1]^2$. If $i \geq 1$, we define 
\[ P_{i+1,j} = \begin{bmatrix}
0 & 0 & 0 & \cdots & 0 \\
\vdots & \vdots & \vdots & \ddots & \vdots \\
0 & 0 & 0 & \cdots & 0 \\
0 & P_{i,1} & P_{i,2} & \cdots & P_{i,k} \\
0 & 0 & 0 & \cdots & 0 \\
\vdots & \vdots & \vdots & \ddots & \vdots \\
0 & 0 & 0 & \cdots & 0 \\
P_{i,j} & 0 & 0 & \cdots & 0
\end{bmatrix}\]
as the block matrix where the slice $0 \; P_{i,1} \; P_{i,2} \; \cdots \; P_{i,k}$ appears on the $j$th row from the bottom. It is easy to see that the induction hypothesis is satisfied. Observe that this is simply a $(k+1)$-by-$(k+1)$ substitution.

Define $X$ as the limit of the patterns $P_{i,1}$. More precisely, take the subshift containing those configurations $x$ such that for all finite $D \subset \Z^2$, $x|_D$ is a subpattern of $P_{i,1}$ for some $i$. Since every nonzero pattern occurring in $X$ occurs in one of the patterns $P_{i,1}$ by definition, we can conclude that its trace is $k$-sparse (by the inductive assumption).

Now we prove\footnote{This mimics the standard argument for the minimality of a primitive (uniform) one-dimensional substitution \cite[Proposition~4.15]{Ku03} and only needs the ``almost primitivity'' of the substitution. The number $3$ in ``$m_{i-3}$-by-$m_{i-3}$'' in the proof plays the role of the usual ``delay until small patterns that eventually appear start appearing''.} almost minimality. Obviously the all-zero point is in the subshift. We next note that every $m_i$-by-$m_i$ block (not necessarily $(0 \mod m_i)$-aligned) that appears in some $P_{i',1}$ is actually contained already in $P_{i+2,1}$. This follows from the fact that all $(0 \mod m_i)$-aligned $2m_i$-by-$2m_i$ blocks in $P_{i+3,1}$ already appear in $P_{i+2,1}$, which is clear from drawing the support of $P_{i+3,1}$. If a $(0 \mod m_i)$-aligned $2m_i$-by-$2m_i$ subpattern $Q$ in $X$ has a nonzero cell, then since every configuration consists of a $(0 \mod m_i)$-aligned grid of empty blocks and $P_{i,j}$-blocks for varying $j$, $Q$ contains some pattern $P_{i,j'}$ entirely. It follows that it contains every pattern $P_{i-1,j''}$ for $j'' \in \{1,2,...,k\}$, thus every $m_{i-3}$-by-$m_{i-3}$ block that appears in $X$ by the previous argument.

The previous argument also shows that, for every $i$, every configuration contains a row with $k$ ones with pairwise distances at least $m_{i-1}$, at a bounded distance from every nonzero symbol. Since $m_{i-1} \rightarrow \infty$, the subshift has essential trace sparseness $k$.

Finally, we observe that $X$ is sofic by Mozes' theorem \cite{Mo89}, since it is defined by a deterministic primitive substitution.
\end{proof}

%

\begin{figure}[h!]
    \begin{center}
\begin{tikzpicture}

\draw[color=black!7!white,step=1.07,xshift=0.5,yshift=0,line width=0.5] (0,0) grid (9.64,6.42);
\draw[color=black!15!white,step=3.21,xshift=0.5,yshift=0,line width=1] (0,0) grid (9.64,6.42);
    \node[anchor=south west,inner sep=0,scale=1.5] at (0,0) {\includegraphics
{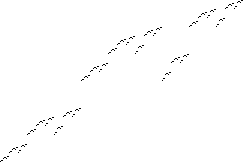}};
    
\end{tikzpicture}


        \end{center}
        \caption{The pattern $P_{5,1}$ for $k = 2$ in Example~\ref{ex:UnboundedRows} (top third cut off).}
        \label{fig:SparseBlobFractal}
    \end{figure}

There are certainly sparse subshifts that are not almost minimal. Thus the assumption of almost minimality is necessary in Theorem~\ref{thm:WhatminimalCases}. The assumption of trace sparseness is also necessary:

\begin{example}
\label{ex:InfiniteComponent}
There is an almost minimal binary sofic $\Z^2$-subshift where the connected component of every $1$ is infinite, but which does not contain a highway in any direction.
\end{example}

\begin{proof}
Use the substitution
\begin{tikzpicture}[baseline = 0]
\fill (0,0) rectangle (0.2,0.2);
\end{tikzpicture}
$\mapsto$
\begin{tikzpicture}[baseline = 0]
\fill (0,0) rectangle (0.2,0.2);
\fill (0,0.2) rectangle (0.2,0.4);
\fill (0.2,0) rectangle (0.4,0.2);
\fill (-0.2,0) rectangle (0,0.2);
\fill (0,-0.2) rectangle (0.2,0);
\end{tikzpicture}
. See Figure~\ref{fig:InfiniteComponent}. Almost minimality is shown as in the previous example. To see the other properties, observe there is an exponential, thus superlinear, lower bound on the number of black cells connected to it in any ball centered at a black cell.
\end{proof}

\begin{figure}[h!]
    \begin{center}
        \includegraphics[scale=0.8]{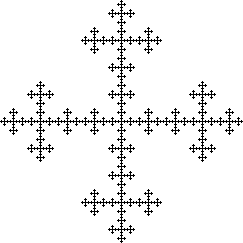}
        \end{center}
        \caption{A pattern from the almost minimal subshift of Example~\ref{ex:InfiniteComponent}.}
        \label{fig:InfiniteComponent}
    \end{figure}

By varying the substitution, it is even possible to make nonzero symbols appear with positive density on every nonzero line (showing that almost minimal subshifts can support nontrivial shift-invariant probability measures).

\begin{example}
For $n \in \N$ define the substitution $\tau_n(0) = 0^{2^n}$ and $\tau_n(1) = 1^{2^n-1} 0$. Consider $w_k = \tau_2(\cdots\tau_k(1)\cdots)$. This word is of length $\prod_{i = 2}^k 2^i$ and the number of $1$-symbols in it is $\prod_{i = 2}^k(2^i - 1)$, so the density of nonzero symbols in it is
\[ \prod_{i = 2}^k (1 - 1/2^i) \geq 1 - \sum_{i = 2}^k 1/2^i > 1/2. \]
Moreover, if $\ell \geq k$, every nonzero symbol occurring in $w_\ell$ is contained in a copy of $w_k$, and thus the ball of radius $|w_k|$ around it has nonzero symbols with upper Banach density at least $1/4$. It follows that in the limit as $k \rightarrow \infty$ we obtain a subshift where every point containing a nonzero symbol contains nonzero symbols with upper Banach density at least $1/4$.
\qee
\end{example}






An SFT containing a road must contain a periodic road by the pigeonhole principle. We now show that sofic shifts can consist of computationally complex ascending paths.

\begin{proposition}
\label{prop:PathsSofic}
Let $X$ be a two-dimensional $\Pi^0_1$-subshift where every nonzero configuration is an $r$-road with uniform ascension constant. Then $X$ is sofic.
\end{proposition}

\begin{proof}
Let $\{0\} \cup \Sigma_+$ be the alphabet of $X$. Let $r$ be such that the nonzero support of every nonzero configuration is up to translation the range of an $r$-path with ascension constant $m$ (so starting from $\bol v$, the path stays above $\bol v$ after $m$ steps). Let $n = rm$.


Let $Y$ be a one-dimensional $\Pi^0_1$-subshift over the alphabet $D = \{\#\} \cup ([-n,n] \times [1,n] \times (\{0\} \cup \Sigma_+)^{[-n,n]^2})$. From such a subshift, we construct a $\Z^2$-subshift where the points of $Y$ give directions for an ascending path, as follows. Let $Z$ be the two-dimensional subshift over $D$ where all rows are constant and columns are points of $Y$. On every non-$\#$ row $i$, if the number in the $[1,n]$-component is $a$, then call $i+a$ the \emph{successor row} of row $i$. Require that the successor of every row is non-$\#$, the rows between a row an its successor contain $\#$, and that every row has a predecessor. Thus, the $[1,n]$-components describe a vertical (strictly) ascending path on the configuration.

Now, add a binary layer where on each row not containing $\#$, at most one cell contains $1$, and rows containing $\#$ are all-$0$. In other words, we overlay the sunny-side-up $X_{\leq 1}$ on the non-$\#$-rows. Call cells containing $1$ on this layer \emph{marked cells}. If on the $Z$-layer the $([-n,n] \times [1,n])$-component contains $\bol u$ at $\bol v$, then $\bol v + \bol u$ is called the \emph{successor} of the cell $\bol v$. We require that $v$ is marked if and only if its successor is marked.

At this point, we have a subshift with at most one single ascending $n$-path on the binary layer, which must travel according to the directions given on the $Z$-layer. Now, consider the factor map $\phi$ defined as follows: If $\bol v$ contains $1$ on the binary layer, then write the nonzero support of the pattern in the $(\{0\} \cup \Sigma_+)^{[-n,n]^2}$-component around $\bol v$. Write $0$ in every remaining cell. We add to our subshift the final SFT constraint that this factor map is well-defined. We obtain a subshift $W$ that has as a factor a subshift over the alphabet $(\{0\} \cup \Sigma_+)$ where the nonzero support of every configuration is the range of a $3n$-path, whose movement is guided by $Y$.

Now, observe that for every $Y$, the subshift $W$ constructed above is a sofic shift. For this, simply observe that $Z$ is sofic by \cite{Ho09,DuRoSh10,AuSa13}, that checking that every row contains at most one binary symbol is doable since $X_{\leq 1}$ is a sofic shift, and that all other constraints we added were local. Of course, then also $\phi(W)$ is sofic.

Finally, we need to show that some $\Pi^0_1$-subshift $Y$ yields $\phi(W) = X$. For this, simply define $Y$ as the $\Pi^0_1$-subshift where for every forbidden pattern of $X$ we forbid every set of directions that would yield a forbidden pattern of $X$ in the $\phi$-image. Then $\phi(W) \subset X$ because any forbidden pattern would have been traced by a path, guided by a word $w$ of $Y$ and then $w$ would have been forbidden in $Y$. On the other hand, for any configuration $x$ of $X$, one can easily find an $n$-path through the configuration and construct the corresponding $(\{0\} \cup \Sigma_+)^{[-n,n]^2}$-patterns, to obtain a legal guiding sequence in $Y$.
\end{proof}

It is easy to construct minimal $\Pi^0_1$ path spaces without a well-defined linear direction, and thus we obtain the following corollary.


\begin{example}
There exists a two-dimensional sofic shift $X$ whose configurations are $r$-roads with uniform ascension rate, but are not periodic and their support does not fit any strip of the form $L + B_r(0)$ where $L$ is a straight line. \qee
\end{example}

Similarly by using Sturmian subshifts with computable angles as the guiding sequences, we obtain sofic shifts where the nonzero support of every configuration does fit a straight strip $L + B_r(0)$, but no such rational line. With a similar idea, we also obtain an example of a sofic shift with a bounded trace, and with only four nondeterministic directions, two of which are irrational.

\begin{example}
\label{ex:BoundedDeterministicSubshift}
Let $X$ be a Sturmian subshift (with any computable irrational angle). Let $Y$ be the two-dimensional sofic shift where each row is constant, and columns are points of $X$. Construct $Z \subset Y \times \{0,1\}^{\Z^2}$ similarly as in Proposition~\ref{prop:PathsSofic}: allow at most one $1$ on each row, and if $(y, z) \in Z$ and $z_{\bol v} = 1$, require $z_{\bol v + (y_{\bol v},1)} = 1$ and that either $z_{\bol v + (0,-1)} = 1$ or $z_{\bol v + (-1,0)} = 1$. See Figure~\ref{fig:Sturmian}.

Now, we claim that the nondeterministic directions are precisely the vertical ones and the ones parallel to the irrational slope.\footnote{We refer to the slope as it appears graphically. This is also the slope of the unique path drawn in $Z$ if angle is interpreted suitably: if $X$ corresponds to a \emph{mechanical word} (see \cite{Lo02}) of slope $\alpha \in (0, 1)$, then the path will always stay at a bounded distance form some translate of the line $\R(\alpha, 1)$.} First, both of them are easily seen to be nondeterministic, since in the vertical directions it is impossible the know the continuation of the $Y$-component, and in a direction orthogonal to the path (as it appears in configurations), it is impossible to know whether the continuation of the half-plane eventually hits the path.

In half-planes of any other slope, we see the color of every column, and thus the contents of the half-plane determines the $Y$-component. If the path visits the half-plane, we can uniquely fill its trajectory based on the $Y$-component, and every path that fits in a strip $L + B_r(0)$ where $L$ is a straight line visits every half-plane not parallel to $L$. \qee
\end{example}

\begin{figure}
\begin{center}
\begin{tikzpicture}[scale = 0.5]
\draw [black!30!white,fill] (0,9) rectangle (18,10);
\draw [black!30!white,fill] (0,8) rectangle (18,9);
\draw [black!30!white,fill] (0,6) rectangle (18,7);
\draw [black!30!white,fill] (0,4) rectangle (18,5);
\draw [black!30!white,fill] (0,3) rectangle (18,4);
\draw [black!30!white,fill] (0,1) rectangle (18,2);
\draw (0,0) grid (18,10);
\node () at (6.5,0.5) {1};
\node () at (6.5,1.5) {1};
\node () at (7.5,2.5) {1};
\node () at (7.5,3.5) {1};
\node () at (8.5,4.5) {1};
\node () at (9.5,5.5) {1};
\node () at (9.5,6.5) {1};
\node () at (10.5,7.5) {1};
\node () at (10.5,8.5) {1};
\node () at (11.5,9.5) {1};
\end{tikzpicture}
\end{center}
\caption{A configuration of the subshift $X$ in Example~\ref{ex:BoundedDeterministicSubshift}. Gray cells denote $1$s of the $Y$-component, $1$s denote the $1$s of the second component.}
\label{fig:Sturmian}
\end{figure}

If we project away the guiding rows $Y$, it is tempting to think that the only nondeterministic directions are the ones parallel to the irrational line, but this is not the case. If the guiding rows are removed, $Z$ becomes sparse, so if it had a deterministic direction then it would have a singly periodic point by Proposition~\ref{prop:DeterministicContainsSinglyPeriodic}, so in fact \emph{every} direction becomes nondeterministic and nonexpansive when the guiding rows are removed. Subshifts with a single irrational nonexpansive direction are not easy to construct \cite{Ho11}, and making such sofic shifts and SFTs is another challenge still \cite{Zi16}. 

We do not believe it is essential that the paths be ascending in Proposition~\ref{prop:PathsSofic}, as every row carries little enough information (the finitely many bounded offsets of finitely many paths) that it can be carried on each row. However, adapting the proof of the theorem above to this general case seems to require additional tricks, if it is possible at all, and perhaps one would need a radically different idea.


\begin{conjecture}
If $X$ is a uniformly trace sparse $\Pi^0_1$-subshift where the nonzero support of every nontrivial configuration is the range of an $r$-path up to translation, then $X$ is sofic.
\end{conjecture}

We also show that sparseness and computability alone are not enough to guarantee soficness, by a slight adaptation of the well-known mirror subshift.

\begin{proposition}
\label{prop:SparseNonSofic}
There exists a $1$-sparse non-sofic $\Pi^0_1$-subshift.%
\end{proposition}

\begin{proof}
We construct such a binary subshift. First, forbid every configuration where the set of rows containing a nonzero symbol intersects every residue class modulo 3, and every configuration where some row contains two ones. If two rows whose distance is not divisible by $3$ contain $1$ at coordinates $(n_1,m_1)$ and $(n_2,m_2)$, respectively, and $m_2 - m_1 \equiv 1 \bmod 3$, then require that $n_1 < n_2$ and that if $|n_1 - n_2| = m$ for some $m$ then there exist two infinite vertical (necessarily nonoverlapping) strips of width at most $m/4$ containing the support of the configuration, and that both of these strips contain a nonzero symbol on exactly every third row. Finally, require that for every pair of adjacent nonzero rows, the horizontal distance between the nonzero symbols on them is the same. This is by definition a $1$-sparse $\Pi^0_1$ subshift (as we explained how to recognize its forbidden patterns). It is non-sofic by a standard exchange argument, see e.g. \cite[Theorem~7.5]{GiRe96} for this technique and references.
\end{proof}


With a slightly more elaborate construction, one can find even a $1$-sparse non-sofic $\Pi^0_1$-subshift where every configuration is either finite or an infinite ascending path (though of course they are not all $r$-paths for any fixed $r$): pick a slowly-growing recursive function $g : \N \to \N$ with $g(n) \rightarrow \infty$ and additionally require that if there are two $1$s at distance at most $g(n)$ from each other, then the two-columns property of the subshift holds with respect to some $m \leq n$.

\subsection{Path covers}
\label{sec:PathCoverExample}

We give an example of extraction of highways through path covers.

\begin{example}
\label{ex:Circles}
Fix some digital approximation scheme that produces subsets of $\Z^2$ as approximations to real Euclidean circles with integer centers and integer radii. We choose for each circle the approximation where in the octant between angles $\pi/2$ and $\pi/4$ (in the standard orientation of the complex plane $\C$ or the unit circle), for each horizontal coordinate $m$ we include the point $(m,n)$ where $n$ is chosen so that the distance to the corresponding octant of the real circle is minimal along the vertical line through $(m,n) \in \Z^2$. Since
\[ \sqrt{r^2 - m^2} = n + \frac{1}{2} \implies r^2 - m^2 = n^2 + n + \frac{1}{4}, \]
the real circle cannot go through $(m,n+1/2)$, and thus the approximation is unique. The rest of the points are produced by dihedral symmetry. This is the scheme produced by the Bresenham circle-drawing algorithm and the midpoint circle algorithm \cite{HeBaCa10}.

Say $x \in \{0,1\}^{\Z^2}$ is a \emph{circles configuration} if there exists a finite set of disjoint circles with centers in $\Z^2$ and integral radii, such that $x_{\bol v} = 1$ if and only if the digital approximation for one of the circles contains $\bol v$, and all the digital approximations are $1$-clusters in $x$. Let $X$ be the smallest subshift containing every circles configuration, call this the \emph{circles subshift}. Since discretizations of circles can have arbitrarily large support in configurations of $X$, Lemma~\ref{lem:MPC} shows that $\MPC_1(C)$ is nonempty.

The proof of Lemma~\ref{lem:MPC} in the case of $X$ is carried out in Figure~\ref{fig:MPC}. Say the \emph{parameters} of a circle are the vector $(a,b,c) \in \Z^2 \times \N$, where $(a,b)$ is the center and $c$ the radius. The subfigures show a series of pairs of circles, one getting larger and surrounding the origin. The circle getting larger always contains the coordinate $(3, 3)$ and it's parameters are given in the captions. The smaller circle has parameters $(-7, -6, 3)$ in each configuration. We trace finite circular paths partially around the big circle to obtain configurations of $\PC_1(X)$. These paths get arbitrarily long for the growing circle, thus in the suitable limit (which looks like the last picture) one can trace a bi-infinite path. Shifting it by $(3, 3)$, we obtain a configuration $x \in \MPC_1(X)$.

\begin{figure}
\begin{center}
\begin{subfigure}[b]{0.45\textwidth} 
\begin{tikzpicture}[baseline = 0, scale = 0.2]
\draw[black!5!white] (0,0) grid (25,25);
\draw[fill, black!10!white] (12,0) rectangle (13,25);
\draw[fill, black!10!white] (0,12) rectangle (25,13);
\foreach \x/\y in { 11/16, 12/16, 13/16, 9/15, 10/15, 14/15, 15/15, 9/14, 15/14, 8/13, 16/13, 8/12, 16/12, 8/11, 16/11, 9/10, 15/10, 4/9, 5/9, 6/9, 9/9, 10/9, 14/9, 15/9, 3/8, 7/8, 11/8, 12/8, 13/8, 2/7, 8/7, 2/6, 8/6, 2/5, 8/5, 3/4, 7/4, 4/3, 5/3, 6/3}{\draw[fill, gray] (\x,\y) rectangle (\x+1,\y+1);}
\foreach \x/\y/\xx/\yy in {16.5/11.5/16.5/12.5, 16.5/12.5/16.5/13.5, 16.5/13.5/15.5/14.5, 15.5/14.5/15.5/15.5, 15.5/15.5/14.5/15.5, 14.5/15.5/13.5/16.5, 13.5/16.5/12.5/16.5, 12.5/16.5/11.5/16.5}{\draw[-{Latex[length=1mm,width=1mm]}] (\x, \y) -- (\xx, \yy);}
\end{tikzpicture}
\caption{$(0, 0, 4)$}
\end{subfigure}
\begin{subfigure}[b]{0.45\textwidth} 
\begin{tikzpicture}[baseline = 0, scale = 0.2]
\draw[black!5!white] (0,0) grid (25,25);
\draw[fill, black!10!white] (12,0) rectangle (13,25);
\draw[fill, black!10!white] (0,12) rectangle (25,13);
\foreach \x/\y in { 4/18, 5/18, 6/18, 7/18, 8/18, 9/18, 10/18, 2/17, 3/17, 11/17, 12/17, 1/16, 13/16, 0/15, 14/15, 15/15, 15/14, 16/13, 17/12, 17/11, 18/10, 4/9, 5/9, 6/9, 18/9, 3/8, 7/8, 18/8, 2/7, 8/7, 18/7, 2/6, 8/6, 18/6, 2/5, 8/5, 18/5, 3/4, 7/4, 18/4, 4/3, 5/3, 6/3, 17/3, 17/2, 16/1, 15/0}{\draw[fill, gray] (\x,\y) rectangle (\x+1,\y+1);}
\foreach \x/\y/\xx/\yy in {18.5/8.5/18.5/9.5, 18.5/9.5/18.5/10.5, 18.5/10.5/17.5/11.5, 17.5/11.5/17.5/12.5, 17.5/12.5/16.5/13.5, 16.5/13.5/15.5/14.5, 15.5/14.5/15.5/15.5, 15.5/15.5/14.5/15.5, 14.5/15.5/13.5/16.5, 13.5/16.5/12.5/17.5, 12.5/17.5/11.5/17.5, 11.5/17.5/10.5/18.5, 10.5/18.5/9.5/18.5, 9.5/18.5/8.5/18.5}{\draw[-{Latex[length=1mm,width=1mm]}] (\x, \y) -- (\xx, \yy);}
\end{tikzpicture}
\caption{$(-5, -5, 11)$}
\end{subfigure}
\begin{subfigure}[b]{0.45\textwidth} 
\begin{tikzpicture}[baseline = 0, scale = 0.2]
\draw[black!5!white] (0,0) grid (25,25);
\draw[fill, black!10!white] (12,0) rectangle (13,25);
\draw[fill, black!10!white] (0,12) rectangle (25,13);
\foreach \x/\y in { 4/24, 5/24, 6/23, 7/22, 8/21, 9/20, 10/19, 11/18, 12/17, 13/16, 14/15, 15/15, 15/14, 16/13, 17/12, 18/11, 19/10, 4/9, 5/9, 6/9, 20/9, 3/8, 7/8, 21/8, 2/7, 8/7, 22/7, 2/6, 8/6, 23/6, 2/5, 8/5, 24/5, 3/4, 7/4, 24/4, 4/3, 5/3, 6/3}{\draw[fill, gray] (\x,\y) rectangle (\x+1,\y+1);}
\foreach \x/\y/\xx/\yy in {23.5/6.5/22.5/7.5, 22.5/7.5/21.5/8.5, 21.5/8.5/20.5/9.5, 20.5/9.5/19.5/10.5, 19.5/10.5/18.5/11.5, 18.5/11.5/17.5/12.5, 17.5/12.5/16.5/13.5, 16.5/13.5/15.5/14.5, 15.5/14.5/15.5/15.5, 15.5/15.5/14.5/15.5, 14.5/15.5/13.5/16.5, 13.5/16.5/12.5/17.5, 12.5/17.5/11.5/18.5, 11.5/18.5/10.5/19.5, 10.5/19.5/9.5/20.5, 9.5/20.5/8.5/21.5, 8.5/21.5/7.5/22.5, 7.5/22.5/6.5/23.5}{\draw[-{Latex[length=1mm,width=1mm]}] (\x, \y) -- (\xx, \yy);}
\end{tikzpicture}
\caption{$(-92, -92, 134)$}
\end{subfigure}
\begin{subfigure}[b]{0.45\textwidth} 
\begin{tikzpicture}[baseline = 0, scale = 0.2]
\draw[black!5!white] (0,0) grid (25,25);
\draw[fill, black!10!white] (12,0) rectangle (13,25);
\draw[fill, black!10!white] (0,12) rectangle (25,13);
\foreach \x/\y in { 5/24, 6/23, 7/22, 8/21, 9/20, 10/19, 11/18, 12/17, 13/16, 14/15, 15/15, 15/14, 16/13, 17/12, 18/11, 19/10, 4/9, 5/9, 6/9, 20/9, 3/8, 7/8, 21/8, 2/7, 8/7, 22/7, 2/6, 8/6, 23/6, 2/5, 8/5, 24/5, 3/4, 7/4, 4/3, 5/3, 6/3}{\draw[fill, gray] (\x,\y) rectangle (\x+1,\y+1);}
\foreach \x/\y/\xx/\yy in {25.5/4.5/24.5/5.5, 24.5/5.5/23.5/6.5, 23.5/6.5/22.5/7.5, 22.5/7.5/21.5/8.5, 21.5/8.5/20.5/9.5, 20.5/9.5/19.5/10.5, 19.5/10.5/18.5/11.5, 18.5/11.5/17.5/12.5, 17.5/12.5/16.5/13.5, 16.5/13.5/15.5/14.5, 15.5/14.5/15.5/15.5, 15.5/15.5/14.5/15.5, 14.5/15.5/13.5/16.5, 13.5/16.5/12.5/17.5, 12.5/17.5/11.5/18.5, 11.5/18.5/10.5/19.5, 10.5/19.5/9.5/20.5, 9.5/20.5/8.5/21.5, 8.5/21.5/7.5/22.5, 7.5/22.5/6.5/23.5, 6.5/23.5/5.5/24.5, 5.5/24.5/4.5/25.5}{\draw[-{Latex[length=1mm,width=1mm]}] (\x, \y) -- (\xx, \yy);}
\end{tikzpicture}
\caption{$(a,b,c) \in P$}
\end{subfigure}
\end{center}
\caption{Illustration of how to obtain nonemptiness of $\MPC_1(X)$ for the circles subshift. In the last subfigure $P = \{(a, a, r) \;|\; a = -\lfloor r/\sqrt{2} \rfloor + 2 , \exists n \in \N: r^2 = 2n^2 - n + 1, r \geq 373 \}$ \cite{Ku79}.}
\label{fig:MPC}
\end{figure}
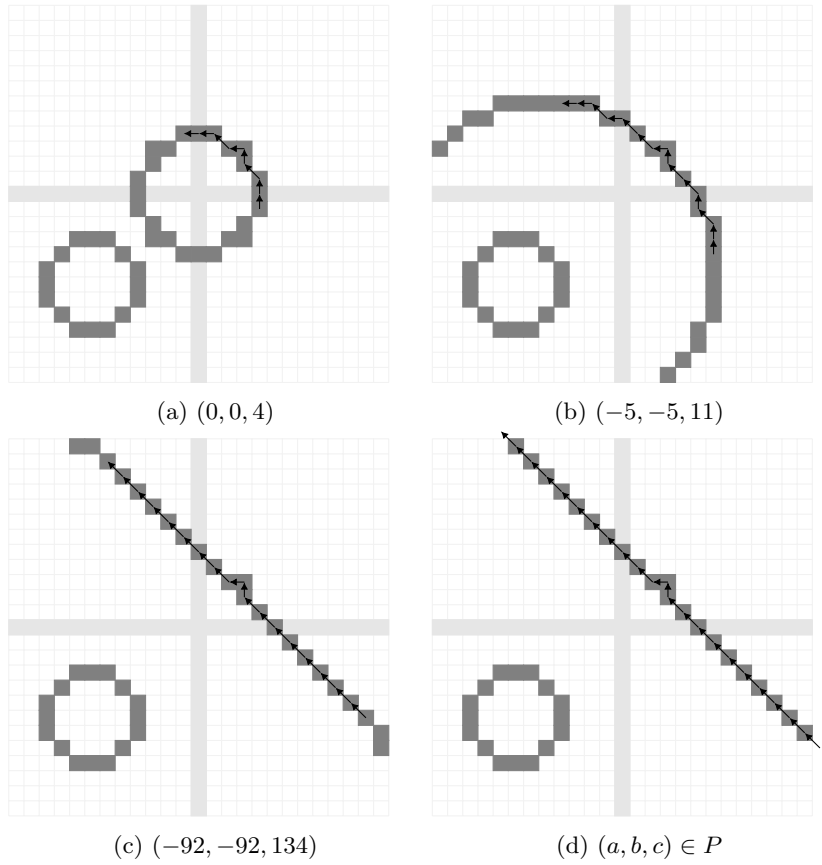

Observe that $x$ is not a highway, as we obtained it as a limit of configurations containing a ``bump'', which appears only once, see \cite{Mo11,Ku79} and OEIS sequence A055979 for details. Note that $x$ is, however, a road by Lemma~\ref{lem:StringLeeway}, as we can trace the path through the finite circle as well. The path proving $7$-roadness is shown in Figure~\ref{fig:MPC2}. (Since $X$ is zero-gluing, we could also simply remove the circle to get a $1$-road.) By passing to a minimal subsystem of $\MPC_7(X)$, we obtain a highway, a diagonal path through the origin. By positioning the circle differently, one can obtain also irrational lines, as constructed in Example~\ref{ex:BoundedDeterministicSubshift}.

\begin{figure}
\begin{center}
\begin{subfigure}[b]{0.45\textwidth} 
\begin{tikzpicture}[baseline = 0, scale = 0.2]
\draw[black!5!white] (0,0) grid (25,25);
\draw[fill, black!10!white] (15,0) rectangle (16,25);
\draw[fill, black!10!white] (0,15) rectangle (25,16);
\foreach \x/\y in { 5/24, 6/23, 7/22, 8/21, 9/20, 10/19, 11/18, 12/17, 13/16, 14/15, 15/15, 15/14, 16/13, 17/12, 18/11, 19/10, 4/9, 5/9, 6/9, 20/9, 3/8, 7/8, 21/8, 2/7, 8/7, 22/7, 2/6, 8/6, 23/6, 2/5, 8/5, 24/5, 3/4, 7/4, 4/3, 5/3, 6/3}{\draw[fill, gray] (\x,\y) rectangle (\x+1,\y+1);}
\foreach \x/\y/\xx/\yy in {25.5/4.5/24.5/5.5, 24.5/5.5/23.5/6.5, 23.5/6.5/22.5/7.5, 22.5/7.5/21.5/8.5, 21.5/8.5/20.5/9.5, 20.5/9.5/19.5/10.5, 19.5/10.5/18.5/11.5, 18.5/11.5/17.5/12.5, 17.5/12.5/16.5/13.5, 16.5/13.5/15.5/14.5, 15.5/14.5/15.5/15.5, 15.5/15.5/14.5/15.5,
14.5/15.5/7.5/8.5,
7.5/8.5/8.5/7.5,
8.5/7.5/8.5/6.5,
8.5/6.5/8.5/5.5,
8.5/5.5/7.5/4.5,
7.5/4.5/6.5/3.5,
6.5/3.5/5.5/3.5,
5.5/3.5/4.5/3.5,
4.5/3.5/3.5/4.5,
3.5/4.5/2.5/5.5,
2.5/5.5/2.5/6.5,
2.5/6.5/2.5/7.5,
2.5/7.5/3.5/8.5,
3.5/8.5/4.5/9.5,
4.5/9.5/5.5/9.5,
5.5/9.5/6.5/9.5,
6.5/9.5/13.5/16.5,
13.5/16.5/12.5/17.5, 12.5/17.5/11.5/18.5, 11.5/18.5/10.5/19.5, 10.5/19.5/9.5/20.5, 9.5/20.5/8.5/21.5, 8.5/21.5/7.5/22.5, 7.5/22.5/6.5/23.5, 6.5/23.5/5.5/24.5, 5.5/24.5/4.5/25.5}{\draw[-{Latex[length=1mm,width=1mm]}] (\x, \y) -- (\xx, \yy);}
\end{tikzpicture}
\end{subfigure}
\begin{subfigure}[b]{0.45\textwidth} 
\begin{tikzpicture}[baseline = 0, scale = 0.2]
\draw[black!5!white] (-1,0) grid (24,25);
\draw[fill, black!10!white] (14,0) rectangle (15,25);
\draw[fill, black!10!white] (-1,15) rectangle (24,16);
\foreach \x/\y in { 5/24, 6/23, 7/22, 8/21, 9/20, 10/19, 11/18, 12/17, 13/16, 14/15, 15/14, 16/13, 17/12, 18/11, 19/10, 20/9,21/8, 22/7, 23/6 }{\draw[fill, gray] (\x,\y) rectangle (\x+1,\y+1);}
\foreach \x/\y/\xx/\yy in { 24.5/5.5/23.5/6.5, 23.5/6.5/22.5/7.5, 22.5/7.5/21.5/8.5, 21.5/8.5/20.5/9.5, 20.5/9.5/19.5/10.5, 19.5/10.5/18.5/11.5, 18.5/11.5/17.5/12.5, 17.5/12.5/16.5/13.5, 16.5/13.5/15.5/14.5, 15.5/14.5/14.5/15.5, 14.5/15.5/13.5/16.5,
13.5/16.5/12.5/17.5, 12.5/17.5/11.5/18.5, 11.5/18.5/10.5/19.5, 10.5/19.5/9.5/20.5, 9.5/20.5/8.5/21.5, 8.5/21.5/7.5/22.5, 7.5/22.5/6.5/23.5, 6.5/23.5/5.5/24.5, 5.5/24.5/4.5/25.5}{\draw[-{Latex[length=1mm,width=1mm]}] (\x, \y) -- (\xx, \yy);}
\end{tikzpicture}
\end{subfigure}
\end{center}
\caption{The path illustrating $7$-roadness of $x$, and a point of $\MPC_7(X)$ (even $\MPC_1(X)$) in its orbit closure under $\tau$.}
\label{fig:MPC2}
\end{figure}
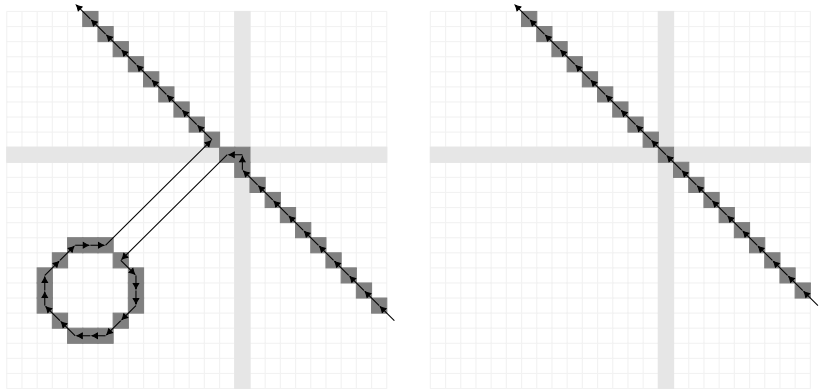 \qee
\end{example}

\subsection{Paths}

In Theorem~\ref{thm:PathCharacterization}, we showed that there are four kinds of minimal path spaces. It is easy to find examples of the first three cases. Path spaces $X$ where all paths are bounded can be found by taking a minimal subshift $Y$ with alphabet contained in $\Z$, and taking the paths of $X$ to be the discrete derivatives (differences between consecutive cells) of points of $Y$.\footnote{Such paths are known as \emph{coboundaries}.} Ascending path spaces can be constructed by summing points of minimal subshifts with alphabet contained in $\Z$ where the sum of every long enough word is positive, and descending path spaces can be found by inverting ascending paths.

The fourth case, where every path is unbounded-to-one and some path is infinite-to-one was not really needed in our characterizations, as we have concentrated on trace sparse spaces. However, this case is quite interesting, and we show by examples that it splits into three subtypes.\footnote{Similar examples can be found in the ascending and descending case when the paths are graphed as two-dimensional configurations and we consider non-horizontal rows.} In the rest of this section, we give representatives of each of these subtypes. All of the examples are substitutive, and thus also give examples of trace sparse sofic shifts when the paths are drawn on $\Z^2$.

The three subcases of minimal path spaces we exhibit are ones where some path visits some cell infinitely many times, and
\begin{itemize}
\item every path visits every cell (that it visits at all) infinitely many times (Example~\ref{ex:Subcase1}, Figure~\ref{fig:Subcase1}),
\item some path visits all cells finitely many times (Example~\ref{ex:Subcase2}, Figure~\ref{fig:Subcase2}), or
\item some path visits some cell finitely many times, but every path visits some cell infinitely many times (Example~\ref{ex:Subcase3}, Figure~\ref{fig:Subcase3}).
\end{itemize}

Write ${\nearrow} = 1$ and ${\searrow} = -1$.

\begin{example}
\label{ex:Subcase1}
There is a minimal subsystem of $\Path$ where every path visits every cell either zero or infinitely many times, and always visits infinitely many cells. Consider the substitution $\tau_1$ defined by
\[ \nearrow \; \mapsto \; \nearrow \nearrow \searrow \searrow \nearrow \nearrow; \;\;\;\; \searrow \; \mapsto \; \searrow \searrow \nearrow \nearrow \searrow \searrow. \]
Let $Y$ be the orbit closure of a two-sided fixed-point of this substitution, considered as a minimal subshift of $\DPath$. Then $X = \phi^{-1}(Y) \subset \Path$ has the property that every path visits every cell (that it visits at all) infinitely many times.

Namely, to each path $w \in \{\nearrow, \searrow\}^\ell$, associate a point $z_w \in \N^\Z$ where $(z_w)_h$ records the number of times the path $w$ visits $h \in \Z$, relative to the starting point of $w$. More precisely, $(z_w)_h = k$ if there are exactly $k$ distinct $j \in [0, \ell]$ such that $\sum_{i = 0}^j w_i = h$. By induction, we see that for $w = \tau_1^n(\nearrow)$ we have $(z_w)_i \geq 2^n$ or $(z_w)_i = 0$ for all $i \in \Z$, and similarly the support of $z_w$ doubles in size after each substitution.

From this, it easily follows that in every path in $X$, every cell (that is visited at all) is visited infinitely many times, and that no path in $X$ is bounded. \qee
\end{example}

\begin{figure}
\begin{center}
\begin{tikzpicture}[scale=0.05]
\draw (0,0) -- ++(1,1) -- ++(1,1) -- ++(1,-1) -- ++(1,-1) -- ++(1,1) -- ++(1,1) -- ++(1,1) -- ++(1,1) -- ++(1,-1) -- ++(1,-1) -- ++(1,1) -- ++(1,1) -- ++(1,-1) -- ++(1,-1) -- ++(1,1) -- ++(1,1) -- ++(1,-1) -- ++(1,-1) -- ++(1,-1) -- ++(1,-1) -- ++(1,1) -- ++(1,1) -- ++(1,-1) -- ++(1,-1) -- ++(1,1) -- ++(1,1) -- ++(1,-1) -- ++(1,-1) -- ++(1,1) -- ++(1,1) -- ++(1,1) -- ++(1,1) -- ++(1,-1) -- ++(1,-1) -- ++(1,1) -- ++(1,1) -- ++(1,1) -- ++(1,1) -- ++(1,-1) -- ++(1,-1) -- ++(1,1) -- ++(1,1) -- ++(1,1) -- ++(1,1) -- ++(1,-1) -- ++(1,-1) -- ++(1,1) -- ++(1,1) -- ++(1,-1) -- ++(1,-1) -- ++(1,1) -- ++(1,1) -- ++(1,-1) -- ++(1,-1) -- ++(1,-1) -- ++(1,-1) -- ++(1,1) -- ++(1,1) -- ++(1,-1) -- ++(1,-1) -- ++(1,1) -- ++(1,1) -- ++(1,-1) -- ++(1,-1) -- ++(1,1) -- ++(1,1) -- ++(1,1) -- ++(1,1) -- ++(1,-1) -- ++(1,-1) -- ++(1,1) -- ++(1,1) -- ++(1,-1) -- ++(1,-1) -- ++(1,1) -- ++(1,1) -- ++(1,-1) -- ++(1,-1) -- ++(1,-1) -- ++(1,-1) -- ++(1,1) -- ++(1,1) -- ++(1,-1) -- ++(1,-1) -- ++(1,1) -- ++(1,1) -- ++(1,-1) -- ++(1,-1) -- ++(1,1) -- ++(1,1) -- ++(1,1) -- ++(1,1) -- ++(1,-1) -- ++(1,-1) -- ++(1,1) -- ++(1,1) -- ++(1,-1) -- ++(1,-1) -- ++(1,1) -- ++(1,1) -- ++(1,-1) -- ++(1,-1) -- ++(1,-1) -- ++(1,-1) -- ++(1,1) -- ++(1,1) -- ++(1,-1) -- ++(1,-1) -- ++(1,-1) -- ++(1,-1) -- ++(1,1) -- ++(1,1) -- ++(1,-1) -- ++(1,-1) -- ++(1,-1) -- ++(1,-1) -- ++(1,1) -- ++(1,1) -- ++(1,-1) -- ++(1,-1) -- ++(1,1) -- ++(1,1) -- ++(1,-1) -- ++(1,-1) -- ++(1,1) -- ++(1,1) -- ++(1,1) -- ++(1,1) -- ++(1,-1) -- ++(1,-1) -- ++(1,1) -- ++(1,1) -- ++(1,-1) -- ++(1,-1) -- ++(1,1) -- ++(1,1) -- ++(1,-1) -- ++(1,-1) -- ++(1,-1) -- ++(1,-1) -- ++(1,1) -- ++(1,1) -- ++(1,-1) -- ++(1,-1) -- ++(1,1) -- ++(1,1) -- ++(1,-1) -- ++(1,-1) -- ++(1,1) -- ++(1,1) -- ++(1,1) -- ++(1,1) -- ++(1,-1) -- ++(1,-1) -- ++(1,1) -- ++(1,1) -- ++(1,-1) -- ++(1,-1) -- ++(1,1) -- ++(1,1) -- ++(1,-1) -- ++(1,-1) -- ++(1,-1) -- ++(1,-1) -- ++(1,1) -- ++(1,1) -- ++(1,-1) -- ++(1,-1) -- ++(1,1) -- ++(1,1) -- ++(1,-1) -- ++(1,-1) -- ++(1,1) -- ++(1,1) -- ++(1,1) -- ++(1,1) -- ++(1,-1) -- ++(1,-1) -- ++(1,1) -- ++(1,1) -- ++(1,1) -- ++(1,1) -- ++(1,-1) -- ++(1,-1) -- ++(1,1) -- ++(1,1) -- ++(1,1) -- ++(1,1) -- ++(1,-1) -- ++(1,-1) -- ++(1,1) -- ++(1,1) -- ++(1,-1) -- ++(1,-1) -- ++(1,1) -- ++(1,1) -- ++(1,-1) -- ++(1,-1) -- ++(1,-1) -- ++(1,-1) -- ++(1,1) -- ++(1,1) -- ++(1,-1) -- ++(1,-1) -- ++(1,1) -- ++(1,1) -- ++(1,-1) -- ++(1,-1) -- ++(1,1) -- ++(1,1) -- ++(1,1) -- ++(1,1) -- ++(1,-1) -- ++(1,-1) -- ++(1,1) -- ++(1,1);
\end{tikzpicture}
\end{center}
\caption{Part of a typical path in Example~\ref{ex:Subcase1}.}
\label{fig:Subcase1}
\end{figure}
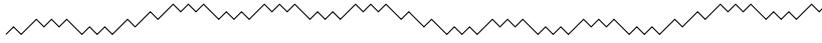

\begin{example}
\label{ex:Subcase2}
Consider the substitution $\tau_2$ defined by
\[ \nearrow \; \mapsto \; \nearrow \nearrow \searrow \nearrow \nearrow; \;\;\;\; \searrow \; \mapsto \; \searrow \searrow \nearrow \searrow \searrow. \]
Define $X$ as in the previous example. From the fact that the total offset $f(n) = 3^n$ of the path $\tau_2^n(\nearrow)$ is unbounded but $f(n)/5^n \rightarrow 0$ and the primitiveness of the subsitution, it is clear that paths in $X$ are not uniformly ascending or bounded. Nevertheless, there is a path in $X$ where every cell is visited a finite number of times. Namely, the limit of $\tau_2^m(\nearrow) . \tau_2^m(\nearrow)$ as $m \rightarrow \infty$ is such a path, where the decimal point denotes the center of the path. \qee
\end{example}

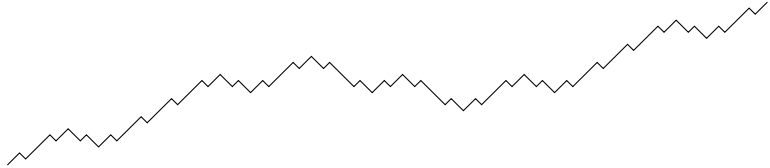
\begin{figure}
\begin{center}
\begin{tikzpicture}[scale=0.08]
\draw (0,0) -- ++(1,1) -- ++(1,1) -- ++(1,-1) -- ++(1,1) -- ++(1,1) -- ++(1,1) -- ++(1,1) -- ++(1,-1) -- ++(1,1) -- ++(1,1) -- ++(1,-1) -- ++(1,-1) -- ++(1,1) -- ++(1,-1) -- ++(1,-1) -- ++(1,1) -- ++(1,1) -- ++(1,-1) -- ++(1,1) -- ++(1,1) -- ++(1,1) -- ++(1,1) -- ++(1,-1) -- ++(1,1) -- ++(1,1) -- ++(1,1) -- ++(1,1) -- ++(1,-1) -- ++(1,1) -- ++(1,1) -- ++(1,1) -- ++(1,1) -- ++(1,-1) -- ++(1,1) -- ++(1,1) -- ++(1,-1) -- ++(1,-1) -- ++(1,1) -- ++(1,-1) -- ++(1,-1) -- ++(1,1) -- ++(1,1) -- ++(1,-1) -- ++(1,1) -- ++(1,1) -- ++(1,1) -- ++(1,1) -- ++(1,-1) -- ++(1,1) -- ++(1,1) -- ++(1,-1) -- ++(1,-1) -- ++(1,1) -- ++(1,-1) -- ++(1,-1) -- ++(1,-1) -- ++(1,-1) -- ++(1,1) -- ++(1,-1) -- ++(1,-1) -- ++(1,1) -- ++(1,1) -- ++(1,-1) -- ++(1,1) -- ++(1,1) -- ++(1,-1) -- ++(1,-1) -- ++(1,1) -- ++(1,-1) -- ++(1,-1) -- ++(1,-1) -- ++(1,-1) -- ++(1,1) -- ++(1,-1) -- ++(1,-1) -- ++(1,1) -- ++(1,1) -- ++(1,-1) -- ++(1,1) -- ++(1,1) -- ++(1,1) -- ++(1,1) -- ++(1,-1) -- ++(1,1) -- ++(1,1) -- ++(1,-1) -- ++(1,-1) -- ++(1,1) -- ++(1,-1) -- ++(1,-1) -- ++(1,1) -- ++(1,1) -- ++(1,-1) -- ++(1,1) -- ++(1,1) -- ++(1,1) -- ++(1,1) -- ++(1,-1) -- ++(1,1) -- ++(1,1) -- ++(1,1) -- ++(1,1) -- ++(1,-1) -- ++(1,1) -- ++(1,1) -- ++(1,1) -- ++(1,1) -- ++(1,-1) -- ++(1,1) -- ++(1,1) -- ++(1,-1) -- ++(1,-1) -- ++(1,1) -- ++(1,-1) -- ++(1,-1) -- ++(1,1) -- ++(1,1) -- ++(1,-1) -- ++(1,1) -- ++(1,1) -- ++(1,1) -- ++(1,1) -- ++(1,-1) -- ++(1,1) -- ++(1,1);
\end{tikzpicture}
\end{center}
\caption{Part of a typical path in Example~\ref{ex:Subcase2}.}
\label{fig:Subcase2}
\end{figure}

\begin{example}
\label{ex:Subcase3}
Consider the substitution $\tau_3$ defined by
\[ \nearrow \; \mapsto \; \nearrow \nearrow \searrow; \;\;\;\; \searrow \; \mapsto \; \nearrow \searrow \searrow. \]

Define $X$ as in the previous examples. As in the above example, it is clear that paths in $X$ are not uniformly ascending or bounded. As in Example~\ref{ex:Subcase1}, consider the configurations $z \in \N^\Z$ corresponding to $\tau^n(\nearrow)$. Then by induction one can show that the support of $z$ is the interval $[0,n+1]$, $z_0 = 1$ and $z_i \geq n$ for $i \in [1,n+1]$. (Seeing $z$ as an element of $\Z^\N$, the recurrence is $z \mapsto z + 2 \cdot 0z0^\omega - 0110^\omega$, with cellwise addition.)

It easily follows that all paths in $X$ have at most one cell that is not visited infinitely many times, as if $m$ is visited at most $k$ times, then for $n > k$, the property of $z \in \N^\Z$ corresponding to $\tau^n(\nearrow)$ or $\tau^n(\searrow)$ established in the previous paragraph shows that $m$ must be the beginning of a $\tau^n(\nearrow)$-path and the end of a $\tau^n(\searrow)$-path (as every path can be written as compositions of such subpaths). Conversely, indeed a path where $0$ is visited once, negative integers are never visited, and positive integers are visited infinitely many times, is obtained as the limit of $\tau^n(\searrow) . \tau^n(\nearrow)$. \qee
\end{example}

\begin{figure}
\begin{center}
\begin{tikzpicture}[scale=0.07]
\draw (0,0) -- ++(1,1) -- ++(1,1) -- ++(1,-1) -- ++(1,1) -- ++(1,1) -- ++(1,-1) -- ++(1,1) -- ++(1,-1) -- ++(1,-1) -- ++(1,1) -- ++(1,1) -- ++(1,-1) -- ++(1,1) -- ++(1,1) -- ++(1,-1) -- ++(1,1) -- ++(1,-1) -- ++(1,-1) -- ++(1,1) -- ++(1,1) -- ++(1,-1) -- ++(1,1) -- ++(1,-1) -- ++(1,-1) -- ++(1,1) -- ++(1,-1) -- ++(1,-1) -- ++(1,1) -- ++(1,1) -- ++(1,-1) -- ++(1,1) -- ++(1,1) -- ++(1,-1) -- ++(1,1) -- ++(1,-1) -- ++(1,-1) -- ++(1,1) -- ++(1,1) -- ++(1,-1) -- ++(1,1) -- ++(1,-1) -- ++(1,-1) -- ++(1,1) -- ++(1,-1) -- ++(1,-1) -- ++(1,1) -- ++(1,1) -- ++(1,-1) -- ++(1,1) -- ++(1,-1) -- ++(1,-1) -- ++(1,1) -- ++(1,-1) -- ++(1,-1) -- ++(1,1) -- ++(1,1) -- ++(1,-1) -- ++(1,1) -- ++(1,1) -- ++(1,-1) -- ++(1,1) -- ++(1,-1) -- ++(1,-1) -- ++(1,1) -- ++(1,1) -- ++(1,-1) -- ++(1,1) -- ++(1,-1) -- ++(1,-1) -- ++(1,1) -- ++(1,-1) -- ++(1,-1) -- ++(1,1) -- ++(1,1) -- ++(1,-1) -- ++(1,1) -- ++(1,-1) -- ++(1,-1) -- ++(1,1) -- ++(1,-1) -- ++(1,-1) -- ++(1,1) -- ++(1,1) -- ++(1,-1) -- ++(1,1) -- ++(1,1) -- ++(1,-1) -- ++(1,1) -- ++(1,-1) -- ++(1,-1) -- ++(1,1) -- ++(1,1) -- ++(1,-1) -- ++(1,1) -- ++(1,1) -- ++(1,-1) -- ++(1,1) -- ++(1,-1) -- ++(1,-1) -- ++(1,1) -- ++(1,1) -- ++(1,-1) -- ++(1,1) -- ++(1,-1) -- ++(1,-1) -- ++(1,1) -- ++(1,-1) -- ++(1,-1) -- ++(1,1) -- ++(1,1) -- ++(1,-1) -- ++(1,1) -- ++(1,1) -- ++(1,-1) -- ++(1,1) -- ++(1,-1) -- ++(1,-1) -- ++(1,1) -- ++(1,1) -- ++(1,-1) -- ++(1,1) -- ++(1,1) -- ++(1,-1) -- ++(1,1) -- ++(1,-1) -- ++(1,-1) -- ++(1,1) -- ++(1,1) -- ++(1,-1) -- ++(1,1) -- ++(1,-1) -- ++(1,-1) -- ++(1,1) -- ++(1,-1) -- ++(1,-1) -- ++(1,1) -- ++(1,1) -- ++(1,-1) -- ++(1,1) -- ++(1,1) -- ++(1,-1) -- ++(1,1) -- ++(1,-1) -- ++(1,-1) -- ++(1,1) -- ++(1,1) -- ++(1,-1) -- ++(1,1) -- ++(1,-1) -- ++(1,-1) -- ++(1,1) -- ++(1,-1) -- ++(1,-1) -- ++(1,1) -- ++(1,1) -- ++(1,-1) -- ++(1,1) -- ++(1,-1) -- ++(1,-1) -- ++(1,1) -- ++(1,-1) -- ++(1,-1);
\draw[color=black!30!white] (0,-1) -- (162,-1);
\end{tikzpicture}
\end{center}
\caption{Part of the path in Example~\ref{ex:Subcase3} where one height is visited only once, and all other heights are visited infinitely many times. The gray line 
marks the height visited only once.}
\label{fig:Subcase3}
\end{figure}
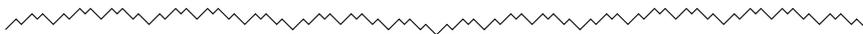

\section{Acknowledgements}

The author would like to thank Pierre Guillon, Mariusz Lema\'nczyk, Kaisa Matom\"aki, Ronnie Pavlov, Siamak Taati, Ilkka T\"orm\"a and Charalampos Zinoviadis for many interesting discussions. We thank an anonymous referee for suggesting many improvements.

\bibliographystyle{plain}
\bibliography{../../../bib/bib}{}

\end{document}